\documentclass[12pt,a4paper]{article}

\usepackage{amsmath,amsthm,enumerate}
\usepackage{graphicx}
\usepackage{amssymb}
\usepackage{float} 
\usepackage{color}
\usepackage{caption}
\usepackage{comment}
\usepackage{pgf}
\usepackage{tikz}
\usetikzlibrary{arrows,automata}
\usepackage[normalem]{ulem}
\newcommand\redout{\bgroup\markoverwith
{\textcolor{red}{\rule[.5ex]{2pt}{0.4pt}}}\ULon}
\usepackage{xcolor}
\usepackage{comment}

\newcommand{\C}{\mathbb{C}}

\newcommand{\R}{\mathbb{R}}
\newcommand{\T}{\mathbb{T}}

\theoremstyle{definition}
\newtheorem{obs}{Remark}[section]

\newtheorem{conj}[obs]{Conjecture}
\theoremstyle{plain}
\newtheorem{prop}[obs]{Proposition}
\newtheorem{teo}[obs]{Theorem}
\newtheorem{lem}[obs]{Lemma}
\newtheorem{cor}[obs]{Corollary}

\newcommand*\colvec[3][]{
    \begin{pmatrix}\ifx\relax#1\relax\else#1\\\fi#2\\#3\end{pmatrix}
}

\def\demP215{{\bf Proof of Proposition 1.6:}\hspace{.2in}}

\title{Julia sets  for Fibonacci endomorphisms of $\mathbb{C}^2$ and  $\mathbb{R}^2$}

\author{S. Bonnot, A. de Carvalho, A. Messaoudi}

\date{}

\begin{document}

\maketitle





\medskip

\centerline{\scshape Sylvain Bonnot}
\medskip
{\footnotesize
 \centerline{Instituto de Matem\'atica e Estat\'istica}
   \centerline{Universidade de S\~ao Paulo, Rua do Mat\~ao
1010}
\centerline{S\~ao Paulo, SP, Brasil}
} 

\medskip

\centerline{\scshape Andr\'e de Carvalho}
\medskip
{\footnotesize
 \centerline{Instituto de Matem\'atica e Estat\'istica}
   \centerline{Universidade de S\~ao Paulo, Rua do Mat\~ao
1010}
\centerline{S\~ao Paulo, SP, Brasil}
} 

\medskip

\centerline{\scshape Ali Messaoudi}
\medskip
{\footnotesize
 \centerline{Departamento de Matem\'atica, Universidade Estadual Paulista}
   \centerline{Rua Cristov\~ao Colombo, 2265, CEP 15054-0000}
\centerline{S\~ao
Jos\'e do Rio Preto-SP, Brasil}
} 

\begin{abstract}
We study the dynamics of the family $f_c(x, y)= (xy+c, x)$ of endomorphisms of 
$\mathbb{R}^2$ and $\mathbb{C}^2$, where $c$ is a real or complex parameter. Such maps 
can be seen as perturbations of the map $f_0(x,y)=(xy,x)$, which is a complexification of the
Anosov torus map $(u,v) \mapsto (u+v,u)$. 

\end{abstract}

\section {Introduction}

The field of mathematics which is now called {\em complex dynamics} was started by Fatou and 
Julia in the beginning of the 20th century: in the late 1910's, they proved a number of results 
about the iteration of polynomial endomorphisms of the Riemann sphere. The field lay essentially 
dormant for several decades until in the early 
1980's, with the advent of computers, it
experienced a vigorous rebirth, became a vibrant area of mathematics which attracted many 
researchers and eventually produced beautiful results with ties to many other areas of
mathematical research (see \cite{Mil} and \cite{cg} for an overview of the area). The study of higher dimensional complex dynamics started in the late 1980's
with the work of Hubbard, Fornaess-Sybony and Bedford-Smillie (see for example \cite{BS}). These and other authors studied 
extensively the complex H\'enon family of polynomial diffeomorphisms of $\C^{2}$. 

Friedland and Milnor (see \cite{FM})
proved early on (1989) that a polynomial diffeomorphism of~$\C^{2}$ with non-trivial dynamics 
is a composition of generalized H\'enon maps.
This article deals with endomorphisms of~$\C^{2}$ which are not diffeomorphisms. More precisely, 
we study the family $f_{c}(x,y)=(xy+c,x)$, where $(x,y)\in\C^{2}$, and $c\in\C$ is a parameter.  Notice that
such maps send the whole line $\{x=0\}$ to the point $(c,0)$, but are diffeomorphisms onto their 
images away 
from this line; also, for real $c$, they can be viewed as endomorphisms of $\R^{2}$. When $c=0$, the 
map $f_{0}(x,y)=(xy,x)$ can be understood as a product, one of whose factors is the linear map 
$(r,s)\mapsto (r+s,r)$, the other being the torus Anosov induced by this same matrix. The dynamics is 
thus easily understood and this is spelled out in Section~\ref{sec:prop}. This also explains 
the choice of name {\em Fibonacci} for the family.  

More general maps of the form $(x,y)\mapsto(xy+c,x+d)$ where studied by Guedj~\cite{G,G0} who 
proved, among other things, that they have a measure of maximal entropy~$\frac{1+\sqrt{5}}{2}$. The 
family $f_{c}$ was also considered in~\cite{MS}, where 
the associated higher dimensional Julia sets 
were related to odometers. In particular, it was shown that the spectrum of the transfer operator 
associated with a stochastic adding machine in an exotic base (given by Fibonacci
numbers) is related to the set $K^{+}(f_c)$, for a real value of $c$ (a result inspired by \cite{kt}). Also, in~\cite{EMSB}, 
various topological properties of certain slices of the sets $K^{+}(f_{c})$ were discussed. Here
$K^{+}(f_{c})$ is the {\em forward filled Julia set} of $f_{c}$, made of all the points  whose forward orbits are bounded: 
$$K^{+}(f_{c}):= \{ z \in \mathbb{C}^2;\, f_{c}^{n}(z), n\ge 0, \mbox { is bounded}\}.$$

In this paper, we start the study of global topological properties of $K^{+}(f_c)$ and of the {\em backward filled Julia set}
\[
K^{-}(f_c) := \{z \in \mathbb{C}^2;\;  f_{c}^{-n} (z) \textrm{ exists }\forall n\ge 0 \text{ and is bounded}\}.
\]
It is shown that, when $0< |c| < 1/4$, $K^{+}(f_c)$ has
infinite Lebesgue measure and when $c<-2$,  $K^{-}(f_c)$ has positive Lebesgue measure.
In the parameter region $0 < c < \frac{1}{4}$, it is possible to describe in greater detail the real slices 
$K^{+}(f_c) \cap \mathbb{R}^2$ and  $K^{-}(f_c) \cap \mathbb{R}^2$. 
It is shown that these sets are finite unions of invariant manifolds of a finite number of 
periodic points. As a consequence, we obtain that, in this parameter range, $K^{+}(f_c) \cap \mathbb{R}^2$ is a connected subset of $\mathbb{R}^2$ and $K^{-}(f_c) \cap \mathbb{R}^2$ is the union of four smooth curves.

\smallskip
Section 2 describes general topological and measure-theoretic properties of the invariant sets 
$K^\pm(f_c)$, valid for any complex parameter $c \in \mathbb{C}$. It also includes the dynamical description of the 
map $f_{0}$ outlined above.
Section 3 concentrates on the case where the parameter $c$ is real, and the map $f_c$ is seen as a self-map of $\mathbb{R}^2$ and the description of the real slices of $K^{\pm}(f_{c})$ as unions of invariant manifolds is given.

\section{Properties of $K^{+}(f_c)$  and $K^{-}(f_c)$}
\label{sec:prop}

This section establishes the main properties of the invariant sets $K^{+}(f_c)$ and $K^{-}(f_c)$.
For any fixed complex number $c$, let us consider the polynomial endomorphism of $\mathbb{C}^2$ defined by

\begin{equation*}
 \begin{array}{rccc}
   f_c :&  \mathbb{C}^2 &\rightarrow & \mathbb{C}^2    \\
    &(x,y)& \mapsto & (xy+c,x).
 \end{array}
\end{equation*}

Observe that $f_c$ is not one-to-one on the set $\{0\} \times \mathbb{C}$ and not onto on the set $ \mathbb{C}  \times \{0\}$.

Consider the maximum norm $\| (x, y) \| =
\max \{\vert x \vert, \vert y \vert\}$ in $\C^{2}$, and define the following $f_{c}$-invariant sets:
\begin{itemize}
\item $K^{+}(f_c) = \{z \in \mathbb{C}^2,\; \sup _{n \in \mathbb{Z}^{+}} \|  f_c^{n} (z) \| < \infty \}$
\item $K^{-}(f_c) = \{z \in \mathbb{C}^2,\;   f_c^{-n} (z) \textrm{ exists for all } n , \text{ and } \sup _{n \in \mathbb{Z}^{+}} \| f_c^{-n} (z) \| < \infty \},$
\item $K(f_c)= K^{+}(f_c) \cap K^{-}(f_c).$
 \end{itemize}
 
Through the article, when no confusion is possible, these sets will be denoted by $K^+_{c},K^-_{c},K_{c}$ or simply by $K^{+},K^{-},K$ when the $c$-dependence is not important or $c$ has been fixed. Likewise, $f_c$ may be denoted simply by $f$.

The next three properties of invariance of these sets follow at once from the definitions just given:
\begin{enumerate}[a)]
\item $ f^{-1} (K^{+})= K^{+},\;  K^{+} - ( \mathbb{C} \times \{0\}) \subset f(K^{+}) \subset K^{+}.$
\item
$f (K^{-})= K^{-},\;  K^{-}  -( \{0\} \times \mathbb{C})  \subset f^{-1}(K^{-}) \subset K^{-}.$
\item
$f^{-1} (K)= K = f (K).$
\end{enumerate}

As a preliminary study, we concentrate now on the simplest case, where $c=0$, as it is useful to build an intuition of the dynamics of the more general cases.

\paragraph{ The case $c=0$.}
In this particular case, the invariant subsets $K^{+}$ and $K^{-}$ have explicit descriptions (see Figure \ref{Kplusczero}). 

Consider the following maps:
\begin{equation*}
 \begin{array}{rccc}
   h_1 : &  \mathbb{R}_{\ge 0}^2 \times \mathbb{T}^2 &\rightarrow& \mathbb{C}^2    \\
    &(r,s,e^{i\alpha},e^{i \beta})  &\mapsto& (re^{i\alpha},se^{i\beta})
 \end{array}
\end{equation*}

and
\begin{equation*}
 \begin{array}{rccc}
   h_2 : &  \mathbb{C}^2 &\rightarrow & \mathbb{R}_{\ge 0}^2 \\
    &(x,y)  &\mapsto &\left(|x|,|y|\right).
 \end{array}
\end{equation*}

The map $h_2$ semi-conjugates $f_0: \mathbb{C}^2 \to \mathbb{C}^2$ to the map 
$\check{f}\colon \R^{2}_{\ge 0}\to\R^{2}_{\ge 0}$, $\check f(r,s)= (rs,r)$. The map $h_{1}$ 
restricts to a homeomorphism 
of $\R^{2}_{>0}\times\T^{2}$ onto $\C^{2}$ and $f_0$ lifts under $h_1$ to the map 
$\hat{f}=\check{f}\times T_A\colon\R^{2}_{\ge 0}\times\T^{2}\to\R^{2}_{\ge 0}\times\T^{2}$, where 
$T_A$ is the linear Anosov map induced by the matrix $\bigl(\begin{smallmatrix}
1&1 \\ 1&0
\end{smallmatrix} \bigr)$, acting on the torus factor.

Observe that $\check{f}|_{\R^{2}_{>0}}$ is conjugated, by taking logarithms of both coordinates, to the 
linear map of $\mathbb{R}^2$ induced by the same matrix. From this, it follows that $\check{f}$ is 
orientation-reversing and has a fixed point at $(1,1)$, which is a hyperbolic saddle. Its stable manifold is 
the branch of hyperbola $\{(r,s)\in\R^{2}_{>0};\, s=r^{-\beta}\}$, where $\beta$ is the golden mean, 
which divides the first quadrant into a left and right parts. 
All points on the left are attracted to $(0, 0)$, which is also fixed by $\check f$, and all points on the right 
go to infinity.  The unstable manifold of $(1,1)$ is $\{(r,s)\in\R^{2}_{>0};\, s=r^{1/\beta}\}$. 
This takes care of the dynamics in the interior of $\R^{2}_{\ge 0}\cup\{(0,0)\}$. On the boundary, the 
positive horizontal axis maps to the 
positive vertical axis which is all sent to the fixed point~$(0,0)$. 

From this analysis, it follows that the only non-wandering dynamics of the product map $\hat f$ 
occur on the fibers over the two fixed points $(0,0),(1,1)$, where the dynamics is the toral 
automorphism $T_{A}$ (as it is on all other fibers, which, however, escape to infinity or are attracted to the fiber over $(0,0)$). Via the semi-conjugacy $h_{1}$, this explains the dynamics of $f_{0}$. The following proposition summarizes the information just discussed.

\begin{prop}
\label{co}
Let $c=0$ and denote by $\beta$ the golden ratio $\beta= \frac{1+ \sqrt{5}}{2}$. The following hold: 
\begin{enumerate}[a)]
\item
$K^{+}_{0}= \left\{(x, y) \in \mathbb{C}^2,\; \vert  y \vert \leq \vert x\vert^{-\beta} \; \right\}$,
\item
$K^{-}_{0}= \left\{(x, y) \in \mathbb{C}^2 \setminus \{(0,0)\},\; \vert y \vert = \vert x\vert^{\frac{1}{\beta}} \; \right\}.$
\end{enumerate}
\end{prop}

\begin{proof}
Let $(F_n)_{n \geq 0}$ be  the {\it Fibonacci sequence} defined by
\[
F_0=F_1=1 \textrm{ and }F_{n+2}=F_{n+1}+ F_{n},\; \forall n \in  \mathbb{Z}^+.
\] 
By induction, one shows easily that iterates of $f_0$ are {\it monomial maps}:
\[
 f^{n}_0(x,y)=\left(x^{F_n} y^{F_{n-1}},x^{F_{n-1}} y^{F_{n-2}}\right) \;  \textrm{ for all } n \geq 1,
\]
and
$$
f^{-n}(x,y)=
 \left\{
\begin{array}{cl}
 (x^{F_{n-1}} / y^{F_{n}}, y^{F_{n+1}} / x^{F_{n}}) & \mbox { if } n \mbox { is odd },\\
 (y^{F_{n-1}} / x^{F_{n-2}}, x^{F_{n-1}} / y^{F_{n}}) & \mbox { if } n \mbox { is even }.
 \end{array}
\right.
$$
Since $F_n = O( \beta^{n})$, both statements follow.
\end{proof}

\begin{figure}[h]
\begin{center}
\includegraphics[scale=0.6]{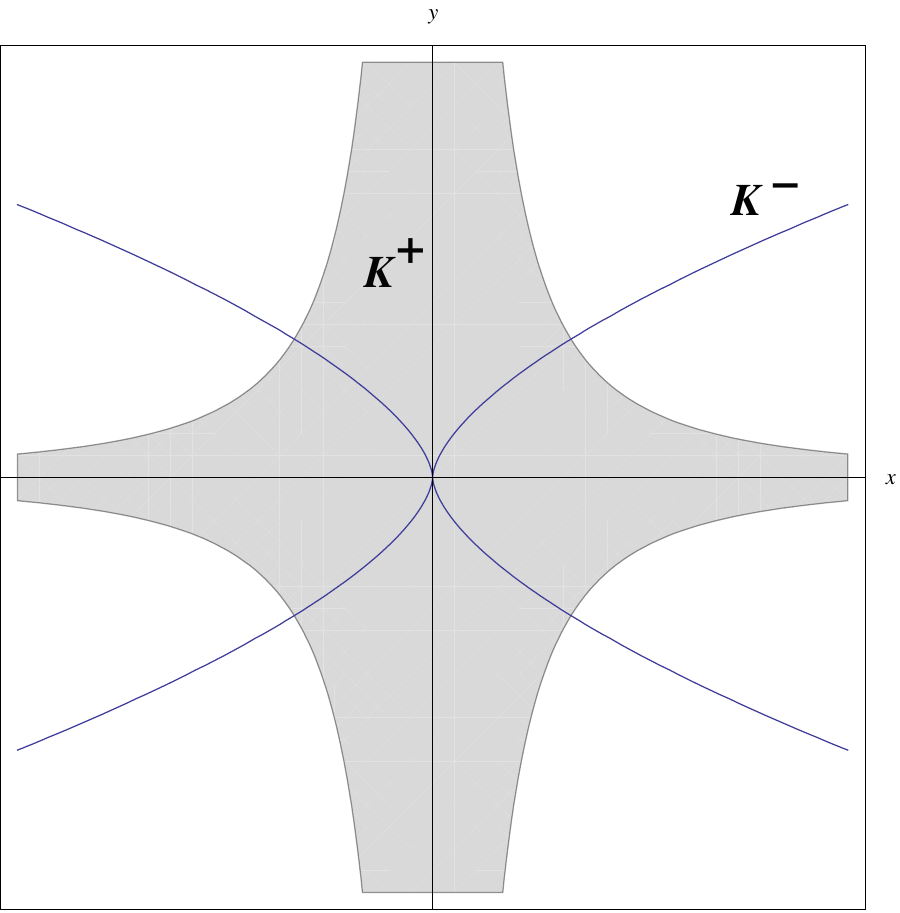}
\caption{The case $c=0$}
\label{Kplusczero}
\end{center}
\end{figure}

\begin{figure}[!ht]
\centering%
\begin{minipage}{0.45\textwidth}
  \includegraphics[width=\linewidth]{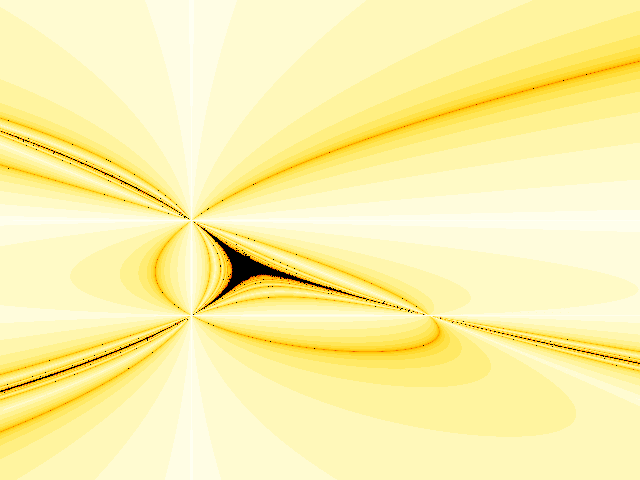}
\end{minipage}%
\begin{minipage}{0.45\textwidth}\centering
  \includegraphics[width=\linewidth]{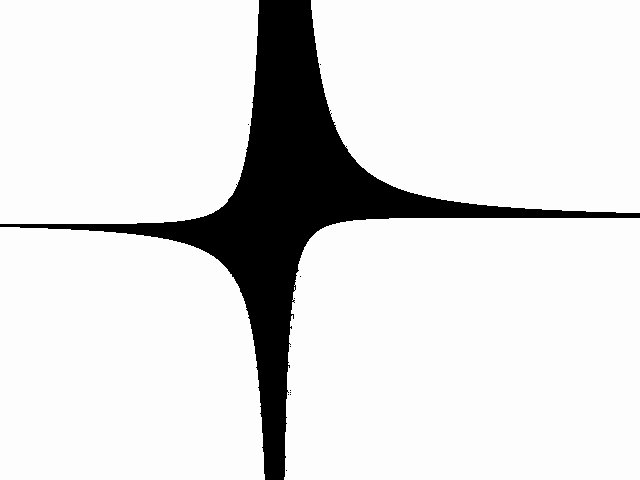}
\end{minipage}%
\caption{ $K^-$ for $c=-3$ (left) \qquad and $K^+$ for $c=-0.6$ (right)}
\end{figure}

\paragraph{Basic properties of the set $K^+(f_c)$.}
We focus now on the filled Julia set  $K^+(f_c)$ and its complement $\mathbb{C}^2 \setminus K^+(f_c)$ called the {\it escaping set}. We show that points that escape to infinity must necessarily escape through a certain set $V_R$, where
\[
V_R= \left\{(x,y) \in \mathbb{C}^2,\; \min \{\vert x \vert, \vert y \vert\} >R\right\}  \]
and the number $R\ge 0$ is chosen appropriately. More precisely, the following will be shown:
\begin{prop}
\label{prop:Kmais0}
Let $R_0=\max\left\{2,\sqrt{2|c|}\right\}$. Then for any $R>R_0$,  $f^{-n}(V_R) \subset f^{-n-1}(V_R)$, 
for every $n \geq 0$, and 
\[
\mathbb{C}^2 \setminus K^{+}=  \bigcup_{n=0}^{+\infty} f^{-n} (V_R).
\]

\end{prop}

\begin{figure}[h]
\begin{center}
\includegraphics[scale=0.35]{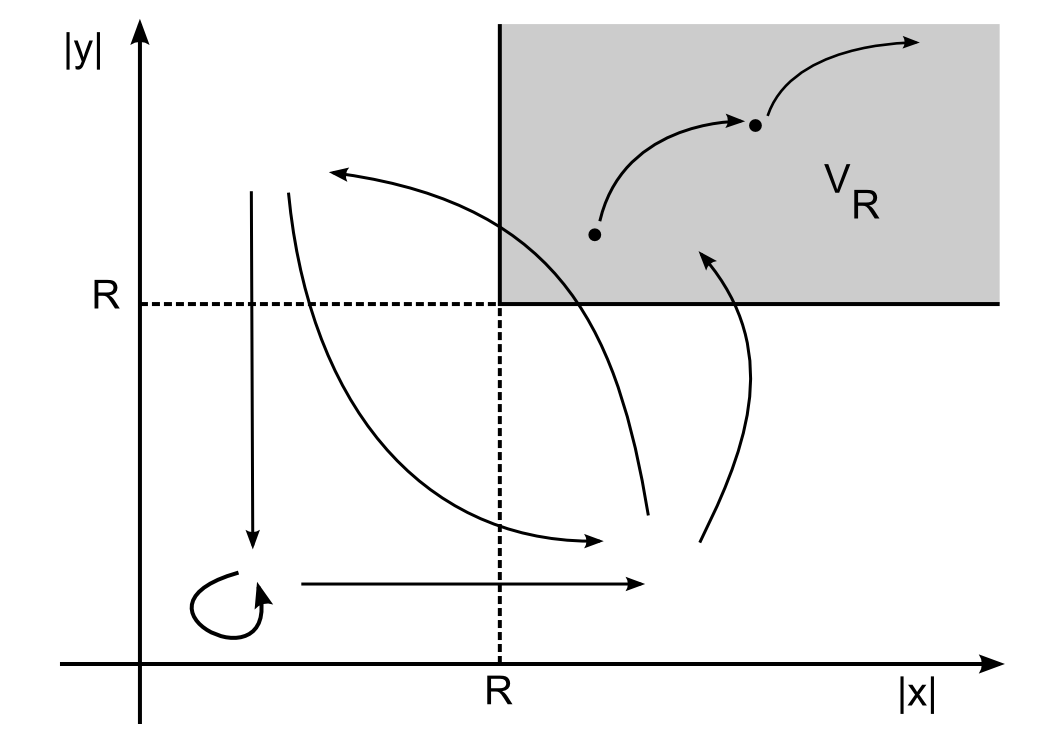}
\caption{Forward escaping points must escape through $V_R$ (Prop. \ref{prop:Kmais0})}
\end{center}
\end{figure}

\begin{obs}
By induction one can show the existence of a sequence of polynomial functions $p_n: \mathbb{C}^2 \to \mathbb{C}$ such that the iterates $f^n_c$, for $n \geq 0$, are given by the formula
\[
f^{n}_{c}(z)= (p_n (z), p_{n-1}(z)), \textrm{ for all } z \in \mathbb{C}^2.
\]
Although the $p_{n}$ also depend on the parameter $c$, we omit it in order not to clutter the notation unnecessarily. 

\end{obs}

The next lemma will be used to prove Proposition \ref{prop:Kmais0}:
\begin{lem}
\label{R0}
For all $R>R_0=\max\left\lbrace2,\sqrt{2|c|}\right\rbrace$, if $\min \left\{\vert p_k(z) \vert,\vert p_{k+1}(z) \vert\right\} > R$ for some integer $k$, then the sequence $(p_ {n}(z))_{n \geq 0}$ is unbounded.
\end{lem}
\begin{proof}

 Let   $R >R_0$ such that $\min \left\{\vert p_k(z) \vert,\vert p_{k+1}(z) \vert\right\} > R$ for 
some integer $k$.
Since $p_{n+1} (z)= p_n (z) p_{n-1}(z)+c$ for all $n \geq 1$, we deduce by the triangle inequality 
that  $\vert p_{k+2}(z)\vert >  R^2 - \vert c \vert > \frac{R^2}{2}$.
Hence 
\begin{equation*}
\vert p_{k+3}(z)\vert >  \frac{R^3}{2}- \vert c \vert > \frac{R}{2} \left(R^2 - \vert c \vert\right) > \frac{R^3}{2^2} .
\end{equation*}
By induction, we then deduce that, for all $n\ge 3$,
\begin{equation*}
\label{fnn}
 \vert p_{k+n}(z)\vert > \dfrac{R^{F_n}}{2^{F_n-1}},
 \end{equation*}
where $F_{n}$ is, as before, the $n$-th Fibonacci number.
 \end{proof}

\begin{proof}[Proof of Proposition \ref{prop:Kmais0}]

Fix $R > R_0$.   It follows from Lemma~\ref{R0} that $\bigcup_{n=0}^{+\infty} f^{-n} (V_R) \subset \mathbb{C}^2 \setminus K^{+}$. Suppose, for a contradiction, that there exists 
$ z= (x, y) \in \mathbb{C}^2 \setminus K^{+}$ 
such that $z \not \in \bigcup_{n=0}^{+\infty} f^{-n} (V_R)$.
Let us consider a large number $a >4$ satisfying $\frac{R^a - \vert c \vert }{R+ \vert c \vert } >1$ and
\begin{equation*}
\label{ccc}
R+ \vert c \vert < R^{a-2} < R^a - \vert c \vert \mbox { and } \| z \| < R^{\frac{a}{2} -1} .
 \end{equation*}
From this one can deduce that $\vert xy+c \vert \leq R^{a-2}+ \vert c \vert < R^a$, thus $\| f(z) \|< R^ a.$
 Let  $n_0 \in \mathbb{N}$ minimal such that
\begin{equation*}
\label{aa}
  R^{a} \leq \| f^{n_0}(z) \|= \| ( p_{n_0}(z),p_{n_{0}-1}(z))  \|.
 \end{equation*}
 Note that $n_0 \geq 2$. On the other hand, we have
 $\| f^{n_{0}}(z) \|=
  \vert p_{n_{0}}(z) \vert $, otherwise
\[ \| f^{n_0 -1}(z) \| \geq \vert p_{n_0 -1}(z) \vert \geq   R^a,\]
which would contradict the fact that $z \not \in \bigcup_{n=0}^{+\infty} f^{-n} (V_R)$.
 Now, since 
\[
R < R^a \leq \vert p_{n_0 }(z) \vert \textrm{ and }  f^{n_0- 1}(z) \not \in V_R,
\]
 we deduce that
$\vert p_{n_{0}+1}(z) \vert \leq R$. Thus by the triangle  inequality, it follows that
$$\vert p_{n_{0}-1}(z) \vert \leq \frac{ R + \vert c \vert}{\vert p_{n_0}(z) \vert}  \leq \frac{ R + \vert c \vert}{R^{a}},$$
and also that
 $$\| f^{n_{0}-2}(z) \| \geq \vert p_{n_{0}-2}(z) \vert \geq \frac{ R^a - \vert c \vert}{\vert p_{n_{0} -1}(z) \vert}  > \frac{R^a ( R^a - \vert c \vert)} {R + \vert c \vert}  \geq R^a.$$
 This last inequality contradicts the minimality of $n_0$.
\end{proof}

Under iteration of a polynomial in $\mathbb{C}$, the set of points with unbounded orbits coincides with the {\it basin of attraction} of infinity in the Riemann sphere $\hat{\mathbb{C}}=\mathbb{C} \cup \{\infty  \}$.
The next corollary says that the same happens in our situation, namely that points with unbounded orbits actually diverge to infinity. 
\begin{cor}
The following properties hold
\begin{enumerate}[a)]
\item
$\mathbb{C}^2 \setminus K^{+}= \{ z \in \mathbb{C}^2,\; \lim \| f^{n}(z) \|= +\infty\}$.
\item
$K^{+}$ is a closed subset of $\mathbb{C}^2$.
\end{enumerate}
\end{cor}
\begin{proof} a) is an immediate consequence of Lemma~\ref{R0} and b) comes from 
Proposition~\ref{prop:Kmais0}.\end{proof}

\begin{figure}[!ht]
\centering%
\begin{minipage}{0.45\textwidth}
  \includegraphics[width=\linewidth]{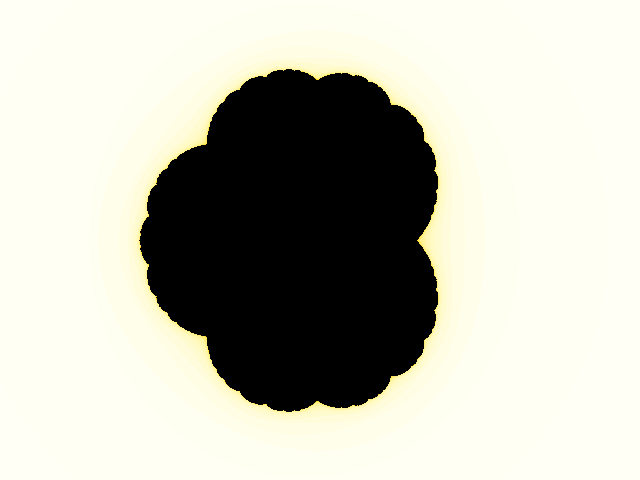}
\end{minipage}%
\begin{minipage}{0.45\textwidth}\centering
  \includegraphics[width=\linewidth]{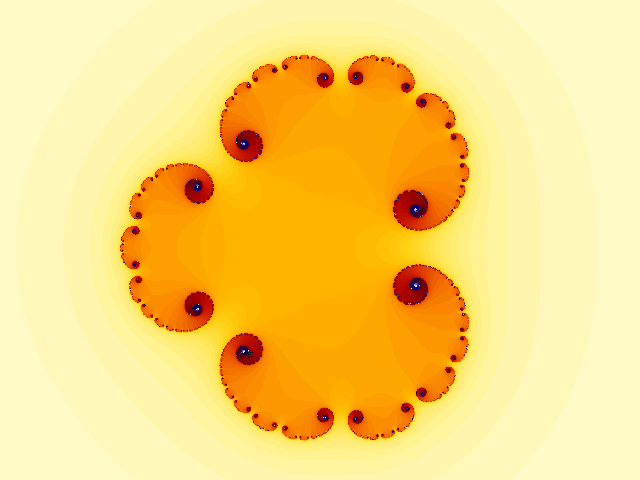}
\end{minipage}%
\caption{Examples of horizontal slices $K^+\cap \{y=0.33\}$ for $c=0.2$ and $c=0.33$}
\label{fig:slice}

\end{figure}

Although the examples given in the figure \ref{fig:slice} show horizontal slices of the set $K^+$ that are bounded, the entire set $K^+$ itself is never bounded. The following proposition is a consequence 
of the results developed in~\cite{EMSB}:

\begin{prop}
\label{unboun}
For all $c \in \mathbb{C}$, $K^{+}(f_c)$ is an unbounded  subset of $\mathbb{C}^2$.
\end{prop}

\begin{proof}
Let $a \in \mathbb{C} \setminus \{0\}$.
In \cite{EMSB}, it is proved that there exists $R >0$ such that the set
$$\{x \in \mathbb{C},\; (x, a) \in K^{+}\}= \bigcap _{n=0}^{+\infty}p_{n}^{-1}\left(\overline{D(0,R)}\right),$$
where the  sequence of compact sets  $p_{n}^{-1} \left(\overline{D(0,R)}\right)$ is decreasing.
Hence $K^{+}$ intersects every complex horizontal line $\mathbb{C} \times \{a\},\; a \ne 0$.
\end{proof}

\paragraph{Basic properties of the set $K^-(f_c)$.}
We perform here a similar study for the set $K^-(f_c)$ with some extra caution, since $f^{-1}_c$ is not globally defined.

For any real $R >0$ define
$$F_R= \left \{(x,y) \in \mathbb{C}^2,\;  0 <\min \{\vert x \vert, \vert y \vert\} < \frac{1}{R} \mbox { and } \| (x, y) \| >R \right \},$$
and
$$G_R= (\mathbb {C} \times \{0 \}) \cup F_R.$$
Observe that
$$ \bigcup_{n=0}^{+\infty} f^{n} (\mathbb {C} \times \{0 \}) \subset \mathbb{C}^2 \setminus K^{-}.$$

The following proposition is the analogue of Proposition \ref{prop:Kmais0} for $K^-$. More specifically, 
we show the existence of a {\it trapping region} for the backward dynamics: points with escaping backwards orbits must escape through the set $F_R$ where they stay trapped.
\begin{prop}
\label{prop:Kmais2}
There exists a  real number $R_1 >0$ such that for all $R > R_1$, we have
$\mathbb{C}^2 \setminus K^{-}=  \bigcup_{n=0}^{+\infty} f^{n} (G_R)$, and $f^{n} (G_R) \subset f^{n+1} (G_R)$ for all $n \geq 0$.
\end{prop}
\begin{proof}
Put $d=2( \vert c \vert +1)$ and choose $R_1$ such that
\begin{equation}\label{eq:R1}
\vert c \vert <  R_1(d-1)/ d^2 < R_1/2 \mbox { and } R_1 > d.
\end{equation}
Let $R > R_1$ and put
 $E= \bigcup_{n=0}^{+\infty} f^{n} (G_R)$.


\begin{figure}[h]
\begin{center}
\includegraphics[scale=0.25]{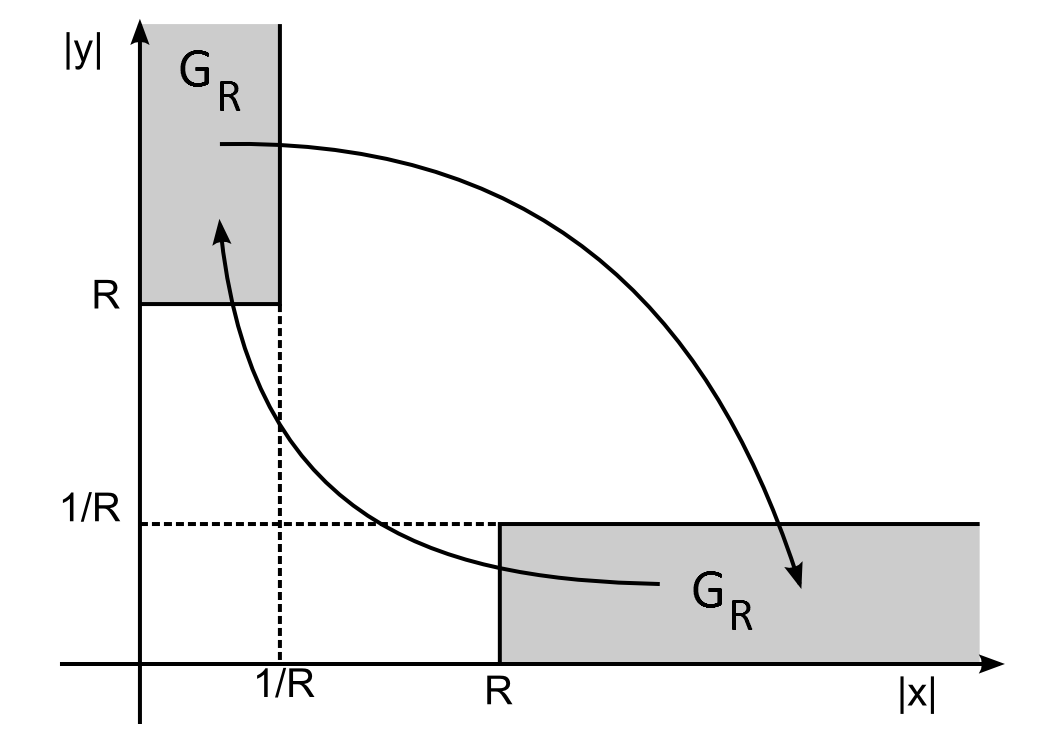}
\caption{Points in $\mathbb{C}^2 \setminus K^-$ must escape through $G_R$ under backward iteration.}
\end{center}
\end{figure}

{\bf Claim 1}:  ${E \subset \mathbb{C}^2 \setminus K^{-}}$.

Let $z \in E$ and assume that $f^{-n}(z) $ is well defined for all integers $n$.  Then $z$ is a positive iterate of some point in $G_R$: there exists a positive integer $n$ and $(x_{-n}, y_{-n}) \in G_R$ satisfying $z=f^{n}(x_{-n}, y_{-n})$.

\medskip
{\it Case 1}:  ${\vert x_{-n} \vert  >R \textrm{ and } 0 < \vert y_{-n} \vert < 1/ R}$.

If this is the case, then $f^{-n-2}(z)= (x_{-n-2}, y_{-n-2})$ satisfies
$$\vert x_{-n-2} \vert = \frac{ \vert x_{-n} - c \vert }{\vert y_{-n} \vert } > (R - \vert c \vert) R> R^2 /2 > R^2/ d$$ and
$$\vert y_{-n-2} \vert = \frac{ \vert ( y_{-n} - c) y_{-n} \vert  }{\vert x_{-n} - c \vert } <
\frac{  (1 / R + \vert c \vert) 1/ R  }{ R- \vert  c \vert } <  d/  R^2.$$
Hence
$$\vert x_{-n-4} \vert = (  R^2/ d - \vert c \vert ) R^2/ d  > R^4/ d^3,$$
since $\vert c \vert <  R(d-1)/ d^2 < R^2(d-1)/ d^2$.
We have also
$$\vert y_{-n-4} \vert < \frac{ ( d / R^2 + \vert c \vert)d/  R^2  }{ R^2 / d- \vert  c \vert }  <  d^3/  R^4.$$
Therefore, we obtain by induction on $k \in \mathbb{N}$ that
$(x_{-n-2k}, y_{-n-2k})= f^{-n-2k}(z)$ satisfies
$$\vert x_{-n-2k} \vert  > d (R/d)^{2k} \mbox { and } \vert y_{-n-2k} \vert  < \frac {1}{d (R/d)^{2k}}.$$
Therefore $z  \in \mathbb{C}^2 \setminus K^{-}$ as was to be shown.

\medskip
 {\it Case 2}: ${0< \vert x_{-n} \vert < 1/R,\; \vert y_{-n} \vert >R}$.

 In this case we show that the preimage of $(x_{-n},y_{-n})$ satisfies the hypothesis of the first case: more precisely, $f^{-n-1}(z)= (x_{-n-1}, y_{-n-1})$ is such that
\[
\vert x_{-n-1} \vert = \vert y_{-n} \vert >R > R/ d
\] and
$\vert y_{-n-1} \vert < d / R$.
From the previous case, we deduce that $\lim_{k \to +\infty} \vert x_{-n-1-2k}\vert = +\infty$. Hence  $z  \in \mathbb{C}^2 \setminus K^{-},$ and we obtain Claim 1.

\bigskip
{\bf Claim 2:} ${ \mathbb{C}^2 \setminus K^{-}\subset E}$.
This second inclusion is very similar to the previous one. 
Indeed, let $z \in \mathbb{C}^2 \setminus K^{-}$ and assume that for all integers $n \in \mathbb{N},\;
f^{-n}(z) \not \in  G_R$ where $ R > R_0$ is a fixed real number.
Let $k_0 \in \mathbb{N}$ such that $\| z \| < R^{k_0}$.
Let $n >k_0 $ be a large integer such that
$R^{F_n -k_0} > 2^{F_n} \mbox { and }  \| f^{-n} (z)\|= \| (x_{-n}, y_{-n})\|  >R^4$.

\medskip
{\it Case 1}: ${\vert x_{-n} \vert > R^4 > R,\; \vert y_{-n} \vert \geq 1/ R}$.

In this case $f^{-n+1} (z)=  (x_{-n+1}, y_{-n+1})$ satisfies
$\vert y_{-n+1} \vert= \vert x_{-n} \vert > R^4 > R $ and $\vert x_{-n+1} \vert > R^ 3 - \vert c \vert> R$.
Hence by (\ref{fnn}), we have $ \| z \| > (R/2)^{F_n} >  R^{k_0}$.
Absurd.

\medskip
{\it Case 2}: ${\vert x_{-n}\vert  \geq 1/ R ,\; \vert y_{-n} \vert  > R^4 > R}$.

Let us show that necessarily $(x_{-n+2},y_{-n+2}) \in V_R$ :
on the one hand we have
\[
\vert y_{-n+2} \vert = \vert x_{-n} y_{-n} + c \vert > R^3 - \vert c \vert > R^3/ 2 > R^2 / 4 >R
\]
and on the other
\[
\vert x_{-n+2} \vert = \vert x_{-n}(x_{-n} y_{-n} + c)+ c  \vert > R^2 / 2 - \vert c \vert > R.
\]
Once again, we conclude by (\ref{fnn}) that $ \| z \| >  R^{k_0}$, which is a contradiction.
This ends the proof of Proposition \ref{prop:Kmais2}.
\end{proof}

\paragraph{Trapping region for $K^+$.}

 The next proposition shows the existence of a {\it trapping bidisk} for $K^+$: in other words, a complex bidisk such that
every point of $K^+$ eventually enters the bidisk under forward iteration.
Such information is crucial for our later goal of describing explicitly the real slice $K^+(f_c) \cap \mathbb{R}^2$ (for $c \in \mathbb{R}$) as a finite union of stable manifolds.

\begin{prop}
\label{prop:Kmais}
Let $D:=D_R= \{(x,y) \in \mathbb{C}^2,\; \max \{\vert x \vert, \vert y \vert\} \leq R \}$.
There exists a  real number $R_2 >0$ such that for all $R > R_2$, we have :
\begin{enumerate}[a)]
\item
$D \cap  K^{+}=  \bigcap_{n=0}^{+\infty}  D \cap f^{-n} (D)$ where $ D \cap f^{-n-1} (D) \subset D \cap f^{-n} (D)$ for all $n \geq 0$.
\item
$K^{+}= \bigcup_{n=0}^{+\infty} f^{-n} (D \cap  K^{+})$  where $f^{-n} (D \cap  K^{+}) \subset  f^{-n-1} (D \cap  K^{+})$  for all $n \geq 0$.
\end{enumerate}
\end{prop}

\begin{figure}[ht]
\center
\includegraphics[scale=0.7]{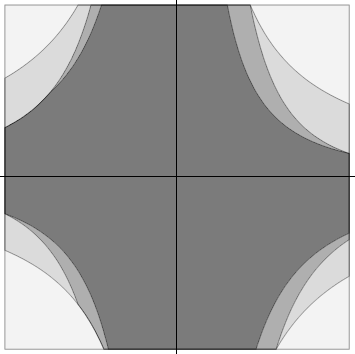}
\caption{Nested sequence $D \cap f^{-n}(D)$ for $n=0,1,2$ and $c=0.3$.}
\label{fig:nested}
\end{figure}

\begin{proof}
{\bf Proof of a)}
Let $ R_1$ be as defined in Proposition \ref{prop:Kmais2}, and let $R >R_1$.

\bigskip
{\bf Claim 1:  ${D \cap f^{-n-1} (D) \subset D \cap f^{-n} (D), \; \forall n \geq 0}$.}

Observe that the property is true for $n=0$ and $n=1$.
Assume that it is true for all integers $1 \leq k \leq n$, and that there exists $z= (x, y) \in D$ satisfying
\[
f^{n+1}(z)= (x_{n+1}, y_{n+1}) \in D \textrm{ and }  f^{n}(z)= (x_{n}, y_{n}) \not \in D.
\]
This implies that $\vert x_n \vert = \vert y_{n+1} \vert \leq R$ and $\vert y_n \vert >R$.
Using the fact that $x_{n+1}= x_n y_n+ c$, we deduce by the triangle inequality that
$ \vert x_n \vert < \frac{R+ \vert c \vert}{R}.$
Therefore
$$\vert x_{n-1} \vert = \vert y_n \vert > R>R/d$$
and 
$$ \vert y_{n-1} \vert = \vert \frac{x_n -c}{y_n} \vert < \frac{ R (1 + \vert c \vert )+ \vert c \vert}{R^2}< d/ R,$$
where $d= 2+ \vert c \vert$.
Thus $$ f^{n-1} (z) \in G_{R/d}.$$
 Let $R_2= d R_1$ and $R>R_2$, we can now deduce from Proposition \ref{prop:Kmais2} that $ G_{R/d}$ is invariant by $f^{-1}$.
This implies that $z \in G_{R/d}$ and $\| z \| > (\frac{R}{d})^n$, which is a contradiction with the fact that $z \in D_R$.

\bigskip
{\bf Claim 2}: ${D \cap  K^{+}=  \bigcap_{n=0}^{+\infty}  D \cap f^{-n} (D)}$.

Let us consider $z \in D \cap  K^{+} $ such that $f^{n}(z)= (x_n, y_n) \not \in D$, for some 
$n \in \mathbb{N}$. From Claim~1, one can assume that $n$ is large.
Suppose without loss of generality that $\vert y_n \vert > R$. Hence by Lemma~\ref{R0}, $\vert x_n \vert \leq R$. Using the triangle inequality, we deduce that
$$\vert x_{n-3} \vert= \vert y_{n-2}\vert = \frac{\vert y_n (y_n -c)\vert} {\vert x_n -c \vert} > \frac{R (R - \vert c \vert)} {(R + \vert c \vert)}= O(R),$$
and also
$$\vert y_{n-3} \vert= \frac{\vert \frac{x_n -c}{y_n} - c_n \vert} {\vert x_{n-3} \vert}  < \frac{ \frac{R + \vert c \vert}{R}+ \vert c \vert} { R { \frac{R - \vert c \vert}{ R+ \vert c \vert}}}= O(1/R) .$$
Thus there exists a positive real number $d$ such that $f^{n-3} (z) \in  G_{R/d}$. Therefore we obtain the inequality $\| z \| > R$, which contradicts the fact that $z \in D$.
 Thus $D \cap  K^{+} \subset  \bigcap_{n=0}^{+\infty}  D \cap f^{-n} (D).$
 The other inclusion is easy to check.
 This ends the proof of Claim 2 and also of Proposition~\ref{prop:Kmais}~a).

\bigskip
\noindent
{\bf Proof of b)}
Let $z \in K^{+}$ and
 $n_0 \in \mathbb{N}$ such that $\| z \| < R^{n_0}.$
Assume that $f^{n}(z) \not \in D_R$ for all integers $n$. Let~$n > n_0$,
 we can suppose that $\vert x_n \vert > R$. Hence by Lemma~\ref{R0}, $\vert y_n \vert \leq R$.
 Now $\vert x_{n-1}\vert = \vert y_n \vert \leq R$. From this one deduces $\vert y_{n-1}\vert >R$ 
 which also means $\vert x_{n-2} \vert > R  $.
 One can then deduce $\vert y_{n-2} \vert < \frac{R+ \vert c \vert}{R}= O(1)$. From that one concludes that
 \[
 \vert x_{n-3} \vert = \vert y_{n-2} \vert  < O(1) \textrm{ and } \vert y_{n-3} \vert \geq \frac {\vert x_{n-2}\vert - \vert c \vert }{\vert x_{n-3}\vert } > \frac{R - \vert c \vert} {O(1)} =  O(R).
 \]
 At this point, one has 
 \[
 \vert x_{n-4}\vert= \vert y_{n-3}\vert > O(R) \textrm{ and } \vert y_{n-4} \vert < \frac{O(1) + \vert c \vert}{O(R)} = O(1/ R).
 \]
 From that last step one can deduce that $f^{n-4}(z) \in G_{R/d}$   where $d >1$ is a positive constant. Hence $\| z \| > R^{n_0}$,  which is a contradiction. 
 Hence
 $K^{+} \subset \bigcup_{n=0}^{+\infty} f^{-n} (D_R \cap  K^{+}).$
 The other inclusion is easy to check since $f^{-1}(K^+)=K^+$.

\bigskip
{\bf Claim 3} : ${f^{-n} (D \cap  K^{+}) \subset  f^{-n-1} (D \cap  K^{+})}$, for all ${n \geq 0}$.

Indeed, it suffices to prove that $D \cap  K^{+} \subset  f^{-1} (D \cap  K^{+})$.
Let $(x_0, y_0) \in D \cap  K^{+}$ and denote $f^{n}(x_0,y_0)= (x_n, y_n)$ for all $n \geq 0$.
Observe that
\begin{eqnarray}
\label{xn}
x_n y_n+ c = x_{n+1},\; \forall n \geq 0.
\end{eqnarray}
Because of Lemma \ref{R0}, if $\max \{\vert x_n \vert, \vert y_n \vert\} > R$ then $\min \{\vert x_n \vert, \vert y_n \vert\} \leq R$.
Suppose that $f(x_0, y_0) \not \in D$.
Hence
\begin{eqnarray}
\label{xt1}
\vert y_1 \vert = \vert x_0 \vert \leq R \mbox { and } \vert x_1 \vert > R.
\end {eqnarray}
Thus
 \begin{eqnarray}
 \label{xt2}
 \vert y_2 \vert = \vert x_1 \vert > R \mbox { and } \vert x_2 \vert \leq R.
 \end{eqnarray}
Therefore
 $\vert y_3 \vert = \vert x_2 \vert \leq  R$.

\medskip
 {\it Case 1}. ${\vert x_3 \vert < R}$.

By (\ref{xt2}) and the fact that $\vert x_2 y_2+c \vert = \vert x_3 \vert <R$,
we deduce that
\begin{eqnarray}
\label{xx2}
\vert x_2 \vert < \frac{R+ \vert c \vert}{R}= O(1).
\end{eqnarray}
By (\ref{xt1}), (\ref{xx2}) and (\ref{xn}) (applied for $n=1$), we have
 \begin{eqnarray}
\label{yy1}
\vert y_1 \vert < \frac{R (1+ \vert c \vert) + \vert c \vert}{R^2}= O(1/R).
\end{eqnarray}
Hence
$\vert x_0 \vert = \vert y_1 \vert < O(1/R)$. Thus $ R < \vert x_1 \vert = \vert x_0 y_0+ c \vert  <
O(1/R) \vert y_0 \vert + \vert c \vert$. Hence
$$\vert y_0 \vert > \frac{R - \vert c \vert}{ O(1/R)} = O(R^2)> R.$$
That is absurd, since $(x_0, y_0) \in D$.

\medskip
  {\it Case 2}. ${\vert x_3 \vert \geq R}$.
  
In this case
\begin{eqnarray}
\label{xx4}
\vert y_4\vert = \vert x_3 \vert \geq R \mbox { and } \vert x_4 \vert \leq R.
\end{eqnarray}
Since $x_4 = x_3 y_3+ c$,
we deduce that
\begin{eqnarray}
\label{xx3}
\vert y_3 \vert < \frac{R+ \vert c \vert}{R}= O(1).
\end{eqnarray}
Hence
$$\vert x_2 \vert = \vert y_3 \vert < O(1) \mbox { and } \vert y_1 \vert= \vert x_0 \vert < \frac{O(1)+ \vert c \vert}{R}= O(1/R).$$
Therefore
\[
\vert y_0 \vert > \frac{R - \vert c \vert}{\vert x_0 \vert}= \frac{R(R - \vert c \vert)}{O(1)+ \vert c \vert} =  O(R^2) > R.
\]
 But this is absurd, since $(x_0, y_0) \in D_R.$
\end{proof}

The following description of the set $K^-$ is very much reminiscent of the description of the unstable manifold $W^u(H)$ of a horseshoe $H$ contained in a box $B$: the intersection $W^u(H) \cap B$ can be described as a decreasing intersection of the subsets of the form $B \cap f^n(B)$, with $n \geq 0$ (see  \cite{MNTU}, section 7.4).

\begin{prop}
\label{Kmoins}
Let $R_2 >0$ be as in Proposition~\ref{prop:Kmais}, i.e., $R_{2}=(2+|c|)R_{1}$, where $R_{1}$ satisfies inequalities~(\ref{eq:R1}). The for all $R > R_2$, we have
\begin{enumerate}[a)]
\item
$D_R \cap  K^{-}=  \bigcap_{n=0}^{+\infty}  D_R \cap f^{n} (D_R)$, where $ D_R \cap f^{n+1} (D_R) \subset D_R \cap f^{n} (D_R)$, for all integers $n \geq 0$.
\item
$K^{-}= \bigcup_{n=0}^{+\infty} f^{n} (D_R \cap  K^{-})$,  where $f^{n} (D_R \cap  K^{-}) \subset  f^{n+1} (D_R \cap  K^{-})$,  for all  integers $n \geq 0$.
\end{enumerate}
\end{prop}

\begin{proof}
Let $ R_2$ as defined earlier and $R >R_2$. Because of Proposition \ref{prop:Kmais}~a), we have
$ D \cap f^{-n-1} (D) \subset D \cap f^{-n} (D), \; \forall n \geq 0$ which implies that
 \begin{eqnarray}
 \label{ffi}
 D \cap f^{n+1} (D) \subset D \cap f^{n} (D), \; \forall n \geq 0.
\end{eqnarray}

\bigskip
{\bf Claim 1}: ${D \cap  K^{-}=  \bigcap_{n=0}^{+\infty}  D \cap f^{n} (D)}$.

Let us consider $z \in D \cap  K^{-} $ such that $f^{-n}(z)= (x_{-n}, y_{-n}) \not \in D$ for some $n \in 
\mathbb{N}$. Because of (\ref{ffi}), we can assume that $n$ is large.
Assume without loss of generality that $\vert y_{-n} \vert > R$. Hence  $\vert x_{-n} \vert \leq R$.
Otherwise, because of
Lemma~\ref{R0}, $z= f^n(x, y) \not \in D$.
We deduce as done in Claim~2 of Proposition~\ref{prop:Kmais} that
$$\vert x_{-n-3} \vert >  O(R),\;
\vert y_{-n-3} \vert < O(1/R) .$$
Thus there exists a positive real number $d$ such that $f^{-n-3} (z) \in  G_{R/d}$. Hence
 $\| z \| > R$ which is absurd.
 From this it follows that
 \[
 D \cap  K^{-} \subset  \bigcap_{n=0}^{+\infty}  D \cap f^{-n} (D).
 \]
 The other inclusion is easy to check.
  This ends the proof of Claim~1 and also of Proposition \ref{Kmoins}~a).

\bigskip
{\bf Claim 2: ${K^{-}= \bigcup_{n=0}^{+\infty} f^{n} (D_R \cap  K^{-})}$ .}

 Let $z \in K^{-}$ and
 $n_0 \in \mathbb{N}$ such that $\| z \| < R^{n_0}.$
Assume that $f^{-n}(z) \not \in D_R$ for all $n$. Let $n > n_0$. Since $x_{-n}= y_{-n+1}$,
 we can suppose that $\vert x_{-n} \vert > R$, hence by Lemma~\ref{R0}, $\vert y_{-n} \vert \leq R$.
 Thus, we obtain (as done in the proof of Claim~2 of Proposition~\ref{prop:Kmais}) that
  $\vert x_{-n-4}\vert > O(R) $ and $\vert y_{-n-4} \vert < O(1/ R)$.
 We deduce that $f^{-n-4}(z) \in G_{R/d}$   where $d $ is a positive constant. Hence $\| z \| > R^{n_0}$, which is absurd.
Therefore
\[
K^{-} \subset \bigcup_{n=0}^{+\infty} f^{n} (D_R \cap  K^{+}).
\]
The other inclusion is easy to check.

\bigskip
{\bf Claim 3}: ${f^{n} (D \cap  K^{-}) \subset  f^{n+1} (D \cap  K^{-}), \forall n \geq 0}$.

We only need to prove that $D \cap  K^{-} \subset  f (D \cap  K^{-})$.
Let $(x_0, y_0) \in D \cap  K^{-}$ such that $f^{-1}(z)= (x_{-1}, y_{-1})\not \in D$.
Then
$$
\vert x_{-1} \vert = \vert y_0 \vert < R \mbox { and } \vert y_{-1} \vert \geq R.
$$
From this it follows at once that
 $$
 \vert x_{-2} \vert = \vert y_{-1} \vert \geq R \mbox { and } \vert y_{-2} \vert < \frac{R+ \vert c \vert}{R}=  O(1).
 $$
Hence
 $\vert x_{-3}\vert <O(1)$ and $\vert  y_{-3} \vert > O(R)$.
 Thus
 \[
 \vert x_{-4}\vert =\vert  y_{-3} \vert   >O(R) \textrm{ and } \vert  y_{-4} \vert < O(1/R).
 \]
We deduce that $f^{-n-4}(z) \in G_{R/d}$ where $d$ is a positive constant.
But this contradicts the fact that $z= (x_0, y_0) \in K^{-}$.

\end{proof}

\begin{prop}
\label{CompactK}
The set  $K= \overline{K^{-} \cap K^{+}}$ is  compact.
\end{prop}

\begin{obs}
With our definition, the set $K^-$, and thus $K$, are not necessarily closed, which is what makes it necessary to prove Proposition~\ref{CompactK}.
\end{obs}

\begin{proof}
 Assume $K^{-} \cap K^{+}$ is not bounded. Take
 $ R_2$ as defined before and $R >R_2$.
 Let $z = (x_0, y_0) \in K^{-} \cap K^{+}$ such that $\| z \| > R^3$.

\medskip
 {\it Case 1:} ${\|z \|=  \vert x_0 \vert}$.

By Propositions~\ref{prop:Kmais} and~\ref{prop:Kmais2}, we deduce that
$$ \frac{1} {R}  \leq \vert y_0 \vert \leq R.$$
Hence $f(x_0, y_0)= (x_1, y_1)$ satisfies
\[
\vert x_1 \vert= \vert x_0 y_0+ c \vert > R^2 - \vert c \vert > R \textrm{ and }\vert y_1 \vert = \vert x_0 \vert > R^3.
\]
 Thus by Lemma \ref{R0},  $ z \not \in K^{+}$, which is a contradiction.

\medskip
{\it Case 2}: ${\| z \|=  \vert y_0 \vert}$.

 We deduce that
$$ \frac{1} {R}  \leq \vert x_0 \vert \leq R.$$
Hence the point $f^{-1}(x_0, y_0)= (x_{-1}, y_{-1})$ satisfies
\[
\vert x_{-1} \vert= \vert y_0 \vert > R^3 \textrm{ and }\vert y_{-1} \vert = \left \vert \frac{ x_0 - c} {y_0}\right \vert < \frac{R+ \vert c \vert }{ R^3} < \frac{1}{R}.
\]
 But this implies that $ z \not \in K^{-}$, which is a contradiction.
\end{proof}

\subsection{The measure of $K^{+}$ and $ K^{-}$ .}
In this section we discuss the Lebesgue measures of the invariant subsets $K^+_{c}, K^-_{c}$ for certain 
value ranges of the parameter $c$. Denote by $\lambda$  the Lebesgue measure on $\mathbb{C}^2$.

\subsubsection{Case where $\vert c \vert $ is small.}

\begin{prop}
 For all $0 \leq \vert c \vert < \frac{1}{4},\;
  \lambda (K^{+}(f_c))= + \infty$.
\end{prop}

\begin{proof}
 By Proposition \ref{prop:Kmais}, we deduce that
$\lambda (K^{+})= \lim \lambda (f^{-n} (D_R \cap  K^{+})).$
Assume that $ \vert c \vert < \frac{1}{4}$, then there exists
$a <1<R$ such that  $\vert c \vert + a^2 < a$.
Hence $f(D_a) \subset D_a$  if we set
 $$D_{a}= \{z \in \mathbb{C}^2,\; \| x \| < a\},$$
which implies $D_a \subset K^{+}.$
 Thus  $(0,0)$ is an interior point of $K^{+}$.
On the other hand if we consider $\Omega_{a}= D_a \setminus \mathbb{C} \times \{0\}$, then
$f^{-1}: \Omega_{a} \to \mathbb{C}^2$ is a well defined  map
and
$$\lambda (f^{-1} (D_a)) \geq \lambda (f^{-1} (\Omega_{a}))= \int_{\Omega_{a}}
\frac{1}{\vert y \vert^2} d \lambda (x,y)= +\infty .$$
\end{proof}

Based on computer investigations we expect the following to be true:
\begin{conj}
For sufficiently large $ \vert c \vert$,
$ \lambda (K^{+}(f_{c})) = 0$.
 \end{conj}

\subsubsection{Case where $c$ is negative real and $\vert c \vert $ is large.}

\begin{prop}
\label{prop:mcgrandenega}
 For all real $ c  <-2,\;
  \lambda (K^{-}_{c})> 0 $.
\end{prop}
\begin{proof}
This is a consequence of Lemma~\ref{lem:mcgrandenega}  below which shows that, in this 
parameter range, $f_{c}$ 
has a repelling fixed point whose (open) basin is contained in $K^{-}_{c}.$
\end{proof}

\begin{lem}
\label{lem:mcgrandenega}
The fixed points of $f^{-1}$ are $(a_1, a_1)$ and $(a_2, a_2)$ where $a_1= \frac{1- \sqrt{1- 4c}}{2}$ and $a_2= \frac{1+ \sqrt{1- 4c}}{2}.$
Moreover,  for all $c <-2$, the fixed point $(a_1, a_1)$ is  an attracting point of $f^{-1}$, and is such that $a_1 <-1$.
\end{lem}

\begin{proof}
The fixed points of $f^{-1}$ are of the form $(x, x)$ where $x^2 -x+c=0$.
On the other hand, the Jacobian matrix of $f^{-1}$ on $(a_i, a_i),\;  i=1,2$ is equal
$$
J_{a_{i}}=
\left (
\begin{array}{cc}
0  &\ 1 \\
\frac{1}{a_i} & -1
\end{array}
\right)
$$
The eigenvalues of $J_{a_{i}}$ are
$$\alpha_{1, a_i}= \frac{-1- \sqrt{1+ \frac{4}{a_{i}}}}{2},\; \alpha_{2, a_i}= \frac{-1+ \sqrt{1+ \frac{4}{a_{i}}}}{2}.$$
 If $a_1 \leq -4$ or equivalently  $c \leq -20$, then  $-1 < \alpha_{1, a_1} < \alpha_{2, a_1}< 0$.
 If  $-4<a_1 <-1$, or equivalently  $-20 < c < -2$, then $\alpha_{1, a_1}$ and  $\alpha_{2, a_1}$ belong to $\mathbb{C} \setminus \mathbb{R}$.
 Since $0 <\alpha_{1, a_1} \alpha_{2, a_1}= -1/ a_1 <1$, we deduce that $\vert \alpha_{1, a_1} \vert =  \vert \alpha_{2, a_1}\vert <1$.
\end{proof}

\bigskip
\noindent{\bf Question}: By Proposition \ref{Kmoins}, we deduce that there exists $R_2>1$  such that for all $R \geq R_2,  \; \lambda (K^{-})= \lim_{n\to+\infty} \lambda ( f^{n} (D_R \cap K^{-}))= \lim_{n\to+\infty}\int_{D_R \cap K^{-}} \vert x_0 x_1 \ldots x_{n-1} \vert^2 d(x, y)$,
where $f^{i}(x, y)= (x_i, y_i)$ for all $i=0,\ldots, n-1$.
 Thus one can ask:
 Is $ \lambda (K^{-}(c))= +\infty$ for $\vert c \vert$ large?














\section{Dynamics in $\mathbb{R}^2$ }

\subsection{Generalities}
\label{subsec:general}
For a parameter $c \in \mathbb{R} $, the map $f_c(x,y)= (xy+c, x)$ can be considered as a self-map of $\mathbb{R}^2$. In this section, we will only consider real parameters $c$ and will, accordingly, 
restrict attention to the dynamics on $\R^{2}$. The invariant sets which will concern us are 
$K^\pm_{\R}(f_{c}):=K^{\pm}(f_{c})\cap\R^{2}$. {\em Since we will not consider subsets of $\C^{2}$ here 
and in order to lighten notation we will denote the sets $K^\pm_{\R}(f_{c})$ simply by $K_{c}^{\pm}$ or $K^{\pm}$.}

Restricting attention to real $c$, a lot more can be said about the topology of the invariant subsets 
$K^\pm$. We show in this section that for a large interval of parameters, the set $K^+$ consists only of the attracting basin of a 3-cycle, together with a finite set of stable manifolds of saddle points on the boundary $\partial K^+$ of $K^+$.

We have seen  that the fixed points of $f$ are $\alpha= (a_1, a_1)$ and $\theta= (a_2, a_2)$ where $a_1= \frac{1- \sqrt{1- 4c}}{2}$ and $a_2= \frac{1+ \sqrt{1- 4c}}{2}.$
Note that the fixed points are in $\mathbb{R}^2$ if and only if $ c \leq \frac{1}{4}$.
The following proposition gives the dynamical types of  $(a_1, a_1)$ and $(a_2, a_2)$  as functions of $c$.

\begin{prop}
\begin{enumerate}[a)]
\item If $ c < \frac{1}{4}$ then $(a_2, a_2)$ is a saddle point of $f$.

\item If $ c < -2$ then $(a_1, a_1)$ is a  repelling fixed point of $f$.


\item If $-2 < c < \frac{1}{4}$ then $(a_1, a_1)$ is an attracting fixed point of $f$.



\item If $c=-2$ then $(a_1, a_1)$ is an indifferent fixed point of $f$ with eigenvalues 
$e^{\pm \frac{2i \pi}{3}}$.

\item If $c=\frac{1}{4}$  then the two fixed points coincide with $( \frac{1}{2}, \frac{1}{2}) $ and the corresponding eigenvalues are $1$ and $\frac{-1}{2}$.

\end{enumerate}
\end{prop}

\begin{proof}
 the Jacobian matrix of $f$ on $(a_i, a_i), i=1,2$ is equal
$$
J_{a_{i}}=
\left (
\begin{array}{cc}
a_i  &\ a_i \\
1 & 0
\end{array}
\right)
$$
The characteristic polynomial of
$J_{a_{i}}$ is $p_i(x)= x^2 - a_i x - a_i$. Hence the eigenvalues of $J_{a_{i}}$ are
$\lambda_{1, a_i}= \frac{a_i- \sqrt{a_i^{2}+ 4 a_i}}{2},\; \lambda_{2, a_i}=  \frac{a_i+ \sqrt{a_i^{2}+ 4 a_i}}{2}$.

Assume $c <0$, hence $a_1 <0$.

\medskip
{\it Case 1:  ${ c <-2}$}.
Then by Lemma \ref{lem:mcgrandenega}, $(a_1, a_1)$ is a repelling fixed point of $f$.
On the other hand, it is easy to check that  $-1 < \lambda_{1, a_2}<0$ and $\lambda_{2, a_2} >1$.
Hence $(a_2, a_2)$ is a saddle point.

\medskip
{\it Case 2:  ${-2 < c <0}$}.

In this case $-1 < a_1 < 0$ and $ 1 < a_2 < 2$.
Then $\lambda_{2, a_1} = \overline{\lambda_{1, a_1}} \in \mathbb{C} \setminus \mathbb{R}$.
Since $\vert \lambda_{1, a_1} \vert^2 = \vert a_1 \vert <1$, we
 deduce that $(a_1, a_1)$ is  an  attracting fixed point of $f$.
On the other hand, it is easy to check that  $-1 < \lambda_{1, a_2}<0$ and $\lambda_{2, a_2} >1$.
Hence $(a_2, a_2)$ is a saddle point of $f$.

\medskip
Now suppose that $0 < c <\frac{1}{4}$.
Hence $0 < a_1 < \frac{1}{2} < a_2 <1$.
Then $-1 < \lambda_{1, a_1}<0,\; 0 <\lambda_{2, a_1} <1,\; -1 < \lambda_{1, a_2}<0,\; \lambda_{2, a_1} >1$.
Thus $(a_1, a_1)$ is an attracting fixed point of $f$ and  $(a_2, a_2)$ is saddle  point of $f$.

The last cases $c=-2$ and $c=\frac{1}{4}$ are simple and we omit the proofs.
\end{proof}

The following proposition spells out a curious fact: the point $(-1,-1)$ is $3$-periodic, for any choice
 of parameter $c$.
\begin{prop}
\label{K1}

For any $c \in \mathbb{R}$, the only $3$-cycle  of $f_c$ is $p=(-1,-1),$
$ f(p)= (1+c, -1),\; f^2(p)= (-1, 1+c)$.
\end{prop}

\begin{proof}Let $(x,y)$ be a $3$-periodic point of $f$  then
$y= (xy+c)x+c$  and $x=(xy+c)y+c$. Hence $x=y$ or $xy= -c-1$.
In the case where $x=y$, we deduce that $x=y \in \{a_1, a_2, -1)$. Hence $p=(-1, -1),\; f(p)= (1+c, -1)$ and $f^2(p)= (-1, 1+c)$  are the points of the cycle of period $3$ of $f_c$.

If $xy= -c-1$, we obtain that $x= -y+c$ and $y^2 + cy +c+1= 0$. But then $y \in \{-1, c+1\}$ and thus $(x, y)= (1+c, -1)$ or $(x, y)= (-1, 1+c)$.
This concludes the proof.
\end{proof}

\begin{obs}
It is easy to check that $f_{c}$ has no 2-cycles (which are not fixed). 
\end{obs}

 \subsection{Case where $0< c <1/4$ }

\subsubsection { Description of $K^{+}$}

In this section, we give a detailed description of the set $ K^+$ for parameters $c \in (0,1/4)$.  In this interval, the map $f_c$ has a single attracting basin, which coincides with the interior of $K^+$. A priori the maps $f_c$ might have many more such attracting basins. Ruling out the existence of some attractor of very high period is what makes the following study quite delicate and technical.


 \begin{teo}
 \label{ws}
For $0< c <1/4$, $K^+$ is the finite union of stable manifolds
 \[
 K^{+}= W^{s}(\alpha) \cup W^{s}(\theta)\cup  W^{s}(p)\cup W^{s}(f(p)) \cup W^{s}(f^2(p)),
 \]
 where $p= (-1, -1),\;  \alpha= (a_1, a_1)$ and $\theta= (a_2, a_2) $.
 Moreover $W^{s}(\alpha)=\mathrm{int} (K^{+})$ and $\partial{K}^{+}= W^{s}(\theta)\cup  W^{s}(p)\cup W^{s}(f(p)) \cup W^{s}(f^2(p))$.

 \end{teo}
 
\begin{figure}[ht]
\center
\includegraphics[scale=0.35]{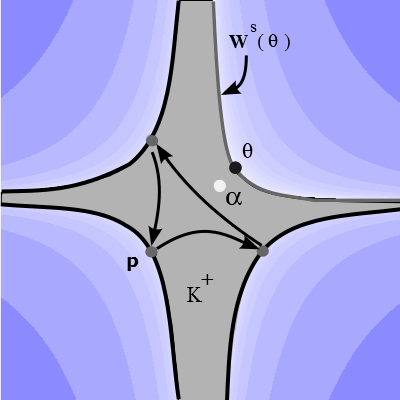}
\caption{Decomposition of $K^+(f_{0.22})$ as a union of stable manifolds}
\label{fig:JuliaSetsReal}
\end{figure}

 The proof will involve a detailed study of the orbits of points and the way they visit some 
 partition (defined below) of the plane in various rectangular regions.

\begin{minipage}{0.5\textwidth}
\begin{figure}[H]
\includegraphics[scale=0.2]{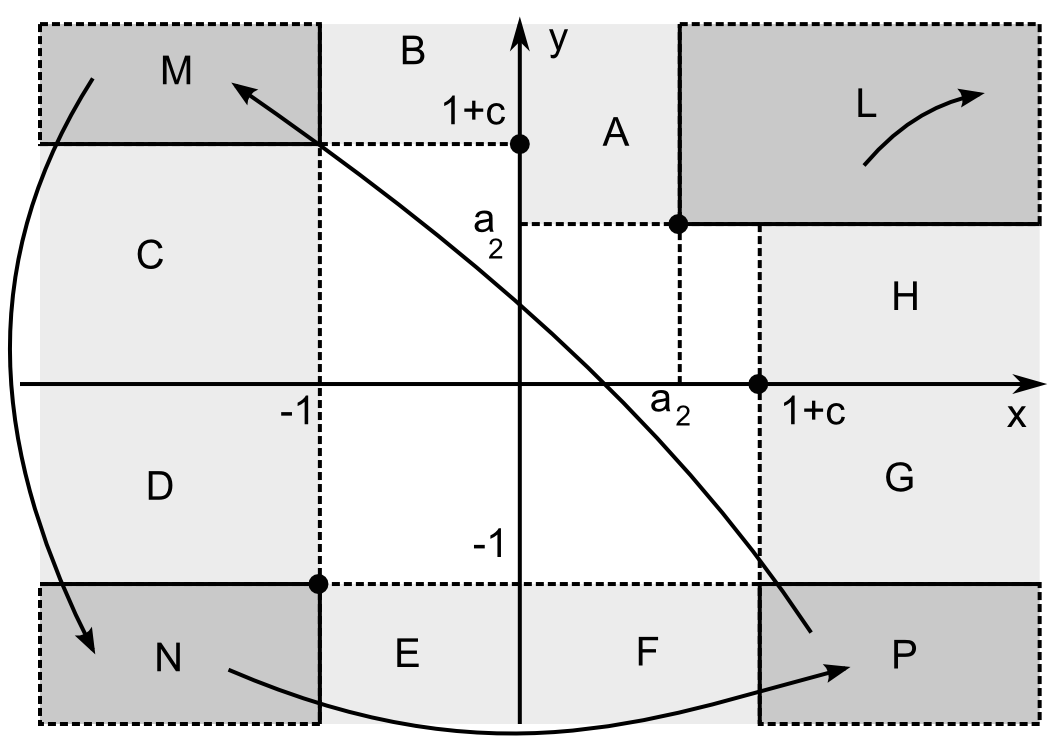}
\caption{\label{fig:LMNP} Rectangular regions $L,M,N,P$}
\end{figure}
\end{minipage} \hfill
\begin{minipage}{0.45\textwidth}
 We define various rectangular regions:
 \begin{enumerate}
 \item[a)] $L= [a_2, +\infty[ \times [a_2, +\infty[,$
 \item[b)] $M= ]- \infty, -1] \times [1+c, +\infty[,$
 \item[c)] $N=  ]- \infty, -1] \times ]- \infty, -1]$
 \item[d)] $P=[1+c, + \infty[ \times ]- \infty, -1],$
\end{enumerate}

\end{minipage}

\bigskip
Let us also introduce $S'= L \cup M \cup N \cup P$ and $S= S' \setminus \{\alpha, \theta, p, f(p), f^2(p)\}$.

\begin{minipage}{0.5\textwidth}
\begin{figure}[H]
\includegraphics[scale=0.2]{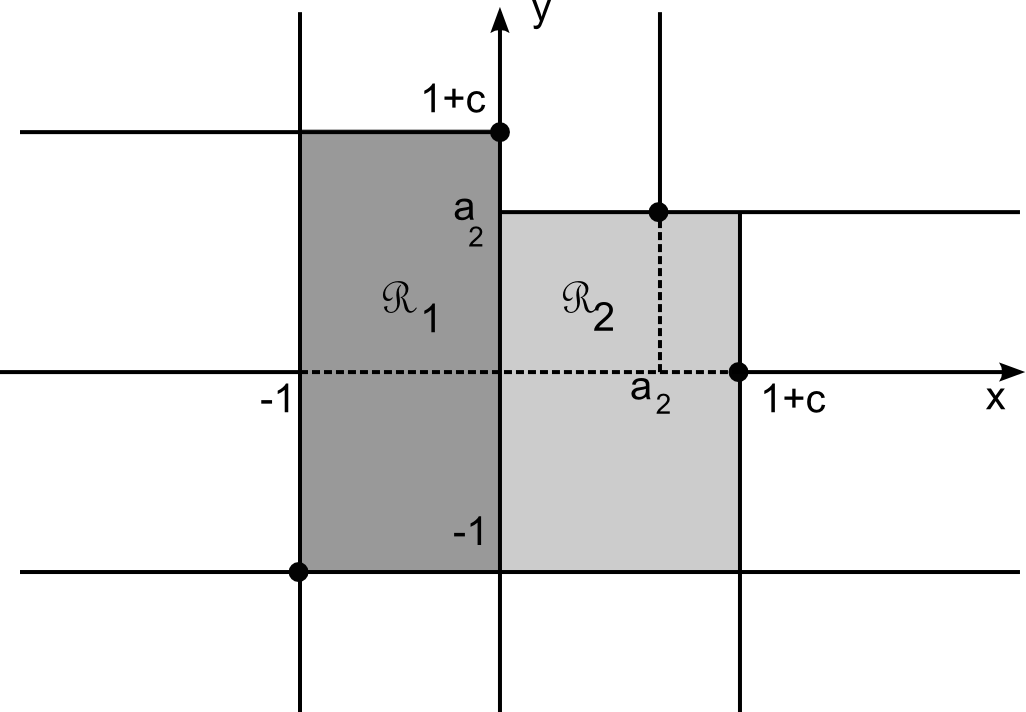}
\caption{\label{fig:R1R2} Rectangles $\mathcal R_1, \mathcal R_2$}
\end{figure}
\end{minipage} \hfill
\begin{minipage}{0.45\textwidth}
Rectangles $\mathcal R_1$ and $\mathcal R_2$:
\begin{enumerate}
\item[a)] ${\cal R}_1= [-1, 0] \times [-1, 1+c]$
\item[b)] ${\cal R}_2= [0, 1+c] \times [-1, a_2]$.
\end{enumerate}

The union is denoted by $\mathcal D:= {\cal R}_1 \cup {\cal R}_2.$

\end{minipage}

Observe that $0 < a_1 < \frac{1}{2}$ and $\frac{1}{2} < a_2 < 1$. Using the partition, one obtains in the next proposition a crude description of the dynamics by analyzing how the different subsets of the partition map into each other (see Figure \ref{fig:LMNP}).

\begin{prop}
\label{prop:creala}
The following properties are valid:
\begin{enumerate}[a)]
\item
$f(L) \subset L,\; f(M) \subset N,\; f(N) \subset P,\; f(P) \subset M.$
\item
For all $(x, y) \in S,\;  \| f^{n} (x, y) \|$ diverges to $\infty$ as $n$ goes to $+\infty$.
\item
$\mathbb{R}^2 \setminus K^{+}= \bigcup_{n=0}^{+\infty} f^{-n}(S) $, where $ f^{-n}(S) \subset f^{-n-1}(S) $ for all integers $n \geq 0$.
\end{enumerate}
\end{prop}

\begin{proof}
a) The proof of the first item is a simple consequence of basic inequalities.
Let $(x,y) \in \mathbb{R}^2$.
\begin{enumerate}[i)]
\item $(x,y) \in L \Rightarrow xy+c \geq a^2+ c = a_2 \Rightarrow f(x,y)= (xy+c, x) \in L$.
\item $(x,y) \in M \Rightarrow xy+c \leq -1 \Rightarrow f(x,y) \in N$.

\item $(x,y) \in N \Rightarrow xy+c \geq 1+c \Rightarrow f(x,y) \in P$.

\item $(x,y) \in P \Rightarrow xy+c \leq -1 \Rightarrow f(x,y) \in M$.

\end{enumerate}

\vspace{0.5em}
\noindent
b) For the next item, the study is more intricate. We proceed by following carefully the itineraries of the points and showing that in many cases some subsequences of the orbits can be shown to be monotone.

Let $(x, y) \in S$, then $(x, y) \in  L \cup M \cup N \cup P \setminus \{\alpha, \theta, p, f(p), f^2(p)\}$.

\vspace{0.5em}

{\it Case 1}:  $(x, y) \in L \setminus \{\theta\}$.

\vspace{0.5em}

{\it Case 1.1}: $\min \{x, y\}> d_0 a_2 $ where $d_0 >1$.

\vspace{0.5em}

{\bf Claim:} For all integers $n \geq 0,$ let us show that $f^{2n} (x, y)= (x_{2n}, y_{2n})$ satisfies
$\min \{x_{2n}, y_{2n}\}> d_n a_2 $
where $d_{n+1}= g(d_{n})$ for all integers $n \geq 0$ and $g(x)= a_2 (x^2-1)+1$.

\vspace{0.5em}

The proof is by induction.
Assume that the claim is true for $n$. Then
$$y_{2n+2}= x_{2n} y_{2n}+c > d_n^2 a_2^2+ c=  (d_n^2-1) a_2^2+ a_2= d_{n+1} a_2$$
and similarly
$$x_{2n+2}= (x_{2n} y_{2n}+c)x_{2n}+ c > d_{n+1}d_{n} a_2^2+ c>  d_n^2 a_2^2+ c= d_{n+1} a_2.$$
And this concludes the proof of the claim.

Now, since $a_2 \geq 1/2$, then the function $g(x)-x $ is non decreasing in $[1, +\infty]$. Hence $(d_n)_{n \geq 0}$ is a non decreasing  sequence.
If $(d_n)$ is bounded, then $(d_n)$
is convergent. Let $l= \lim  d_n$, then $l \in \{1, \frac{1}{a_2}-1\}$.
Hence $l  \leq 1$, which is absurd, since $d_n >1$ for all integers $n \geq 0$.
Thus $\lim d_n = +\infty$. On the other hand, for all integers $n \geq 0$ one has
\[
  x_{2n+1}= y_{2n+2} > d_{n+1}a_2 \textrm{ and } y_{2n+1}= x_{2n} > d_{n}a_2.
  \]
We deduce from this that
$\lim x_{n}= \lim y_n= + \infty$.

\vspace{0.5em}

{\it Case 1.2: $\min \{x, y\}=  a_2 $}

Since $\max\{x, y\} > a_2$, we deduce that there exists a real number $d_0 >1$ such that 
$\min\{x_2,y_2\}> d_0 a_2 $, and then one can conclude as in case 1.1.

\vspace{0.5em}

{\it Case 2}:  $(x, y) \in N \setminus \{p\}, i.e., \max\{x, y\} \leq -1$ and $(x, y) \ne (-1, -1)$.

\vspace{0.5em}

{\it Case 2.1}: $\max\{x, y\} < -1$.

\vspace{0.5em}

{\bf Claim:}  $\forall n \geq 1,\; f^{3n} (x, y)= (x_{3n}, y_{3n})$ with
$x_{3n} < x_{3(n-1)}< -1$ and  $y_{3n} < y_{3(n-1)}<-1$.

\vspace{0.5em}

Indeed, we have
\begin{eqnarray}
\label{hg}
y_{n+3}= (x_ {n} y_{n} +c)x_{n}+c,\; x_{n+3}= y_{n+3}(x_ {n} y_{n} +c)+c,\; \forall n \geq 0.
\end{eqnarray}

If $\max (x_n, y_n) <-1$, then $x_n y_n + c > -y_n +c>0$. Thus $(x_n y_n + c)x_n  < y_n -c$.
Hence $y_{n+3} < y_n$.
On the other hand,
\[
 x_n y_n + c > -x_n + c \textrm{ and } (x_n y_n + c) y_n+ c <  x_n <-1.
 \]
Hence $ x_{n+3} < (x_n-c)+ c= x_n.$
Then, the claim holds.

If $f^{n}(x, y)$ is bounded, then the sequences $(x_{3n})_{n \geq 0}$ and $(y_{3n})_{n \geq 0}$ are convergent.
Let $l= \lim x_{3n} < -1$ and $l'=  \lim y_{3n} < -1$.
By (\ref{hg}), we deduce that
$$(ll'+c)l+c= l',\;  (ll'+c)l'+c= l.$$
Hence $(ll'+c+1) (l-l')=0$. Since $ll'+c+1>0$, we have
 $l=l'$, then $l^3+(c-1)l+c=0$.
Thus $l \in \{-1, a_1, a_2\}$, which is absurd, since $l <-1$.
Hence $f^{3n} (x, y)$ converges to $(-\infty, -\infty)$.
Thus $\lim f^{3n+1} (x, y)=(+\infty, -\infty)$ and $\lim f^{3n+2} (x, y)=(-\infty, +\infty)$.

\vspace{0.5em}

{\it Case 2.2}: $\max\{x, y\} = -1$ and $\min\{x, y\} < -1$.

Thus as above, $x_3 < x \leq -1$ and $y_3 < y \leq -1$, and as in case 2.1, we are done.
In the  cases  $(x, y) \in M$ or  $(x, y) \in P$, we are done because of Proposition~\ref{prop:creala}~a) 
and the fact that b) is true if $(x, y) \in N\setminus \{p\}$.

\medskip\noindent
c) For the proof of item c), it suffices to prove that $\mathbb{R}^2 \setminus K^{+} \subset \bigcup_{n=0}^{+\infty} f^{-n}(S) $.
This comes from the fact that for a real number $R>0 $ sufficiently large, we have $\mathbb{C}^2 \setminus K^{+}=  \bigcup_{n=0}^{+\infty} f^{-n} (V_R)$ (see item 1 of Proposition \ref{prop:Kmais0}). If we choose by  $R > 1+c$, then $V_R\subset S$ and we are done.
\end{proof}

\begin{obs}
\begin{enumerate}[a)]
\item
If $(x, y) \in S$ and we use the notation $f^{n} (x, y)= (x_n, y_n) $ for $n \geq 0$, then $\min \{\vert x_n \vert , \vert y_n \vert\}$ diverges to $\infty$ as $n$ goes to $+\infty$.
\item
If $-1 < c <0$, then Proposition \ref{prop:creala} is true and the proof is the same. In this case $\frac{1 - \sqrt{5}}{2} < a_1 < 0$ and $1 < a_2 < \frac{1 + \sqrt{5}}{2}$.
\end{enumerate}
\end{obs}

\subsubsection {Dynamics of $f$ inside ${\cal D} \cap  K^{+}$}

The dynamics can be analyzed in more detail by introducing a finer partition on the set 
$\mathcal{D}$, and studying the itineraries of the points in $K^+$ in relation to this partition.

\begin{minipage}{0.5\textwidth}
\begin{figure}[H]
\includegraphics[scale=0.5]{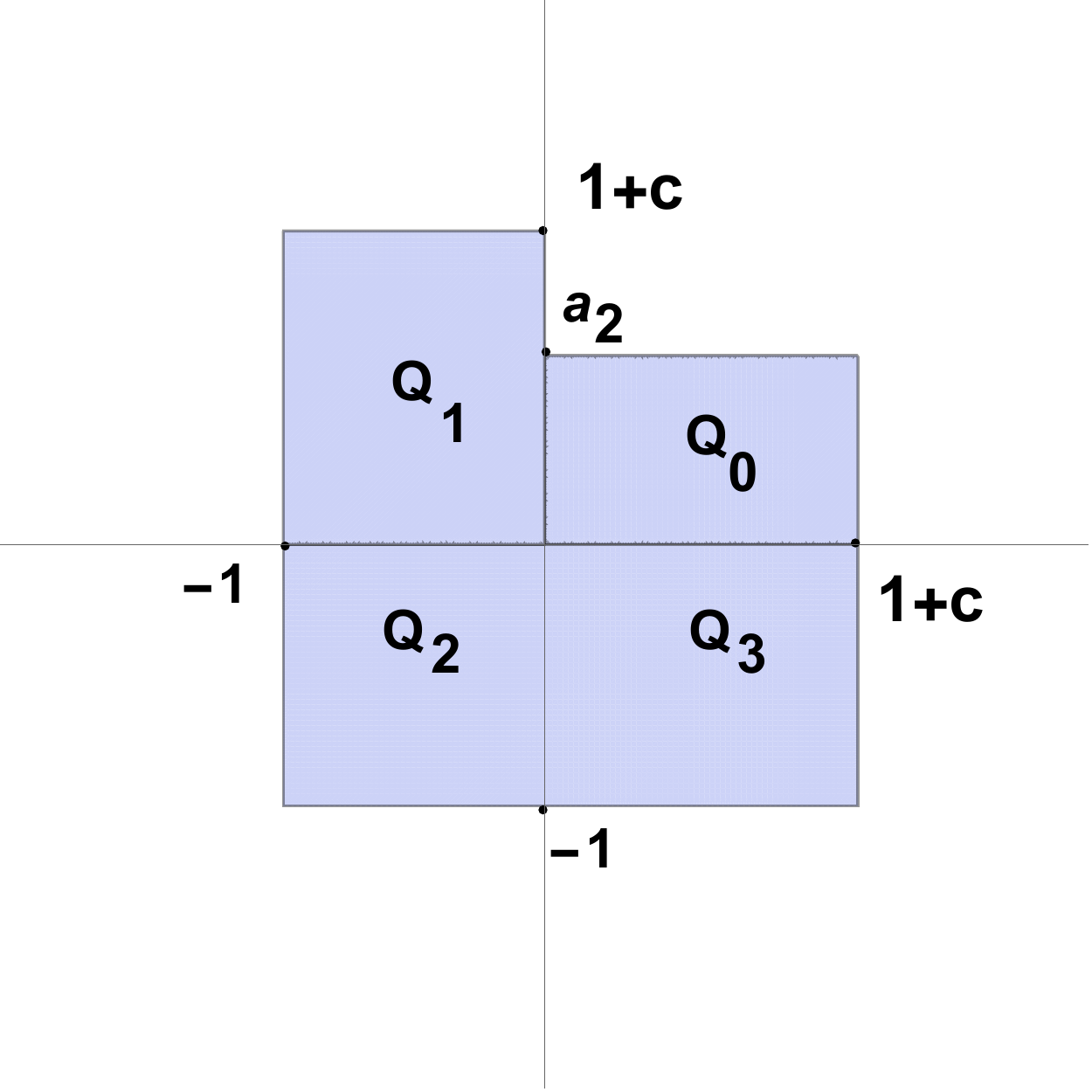}
\caption{Rectangles  $Q_0,Q_1,Q_2,Q_3$}
\end{figure}
\end{minipage} \hfill
\begin{minipage}{0.43\textwidth}

Therefore let us introduce the following rectangular regions:
$$Q_0= [0, 1+c] \times [0, a_2],\; Q_1= [-1, 0] \times [0, 1+c ],$$
$$Q_2= [-1, 0] \times [-1, 0],\; Q_3= [0, 1+c] \times [-1, 0 ],$$
$$A= [0,a_2] \times [a_2, +\infty[,\;   H= [1+c,+\infty[ \times [0, a_2].$$

\end{minipage}

\vspace{0.3cm}

\begin{prop}
\label{hj}
$\bigcup_{n=0}^{+\infty} f^{-n} ({\cal D} \cap  K^{+}) \subset W^{s}(\alpha) \cup W^{s}(\theta) $.
\end{prop}

We will need the following simple fact:

\begin{lem}
\label{ca}
$ (1+c)a_2 \leq 1$.
\end{lem}
\begin{proof}
$a_2 (1+c)= \frac{1+ \sqrt{1+4c}}{2} (1+c) \leq 1$ if and only if $c^2(c+2) \geq 0$.
Since $c \geq -2$, we obtain the result.
\end{proof}

\begin{obs}

As a consequence of the previous Lemma, we have
 $c (1+c) <1$ and $a_2 c <1$.

\end{obs}

The dynamics induced on the partition can be described as follows:

\begin{lem}
\label{insid}
The following inclusions hold:
\begin{enumerate}[a)]
\item
$f(Q_0) \subset Q_0 \cup A \cup L$.
\item
$f(Q_1 ) \subset Q_2 \cup Q_3$.
\item
$f(Q_2 ) \subset Q_3$.
\item
$f(Q_3 ) \subset Q_0 \cup Q_1 \cup A$.
\item
$f(A) \subset Q_0 \cup H$.
\item
 $f(H ) \subset A \cup L$.
\end{enumerate}
\end{lem}

\begin{proof}
\begin{enumerate}[a)]
\item Let $z \in Q_0  = [0, 1+c] \times [0,a_2]$.
Then  $f(z) \in [c, a_2 (1+c)+ c]\times [0, 1+c] $. By Lemma \ref{ca}, we have $a_2 (1+c) \leq 1$. 
From this we deduce that
$$f(Q_0) \subset [c, 1+ c]\times [0, 1+c] \subset Q_0 \cup A \cup L.$$

\item $f(Q_1 ) \subset [-1, c] \times [-1, 0] \subset Q_2 \cup Q_3$.

\item $f(Q_2 ) \subset [c, c+1] \times [-1, 0] \subset Q_3$.

\item $ f(Q_3 ) \subset [-1, c] \times [0, 1+c] \subset Q_0 \cup Q_1 \cup A$.

\item $f(A) \subset [0, +\infty[ \times [0, a_2] \subset Q_0 \cup H$.

\item $ f(H ) \subset [c, +\infty] \times [1+c, +\infty] \subset A \cup L$.
\end{enumerate}
\end{proof}

By Lemma~\ref{insid}, we see that if the orbit of $z= (x,y)$ enters in $\mathcal D \cap K^{+}$, then it must enter $Q_0$ or $Q_3$. In the first case, either the orbit stays in $Q_0$ 
(and then $\lim f^n(z)= \alpha$), or the orbit alternates between $Q_0$ and $A$ (and then 
$\lim f^n(z)= \theta$, see Proposition~\ref{prop:r0a}).
Now if the orbit enters $Q_3$ without ever entering $Q_0$, we will show 
(Proposition~\ref{r1}) that there exists a subsequence $n_k$ such that 
$f^{n_k}(x,y)= (x_{n_k}, y_{n_k}) \in Q_{3}$ with the property that $n_0=0, n_{k+1}- n_k \in \{2, 3\}$ and such that $x_{n_k}$ is decreasing and $y_{n_k}$ is increasing, which leads to a contradiction.

With this preliminary analysis of the induced dynamics, we can now determine the fate of points in $Q_0 \cap K^+$ under forward iteration:

\begin{prop}
\label{prop:r0a}
Let $z \in Q_0 \cap K^{+} \setminus \{\theta\}$, then the following properties are valid:
\begin{enumerate}[a)]
\item
If $f(z) \in Q_0$, then $f^n (z) \in Q_0$ for all integers $n \geq 2$. In this case $\lim f^{n}(z)= \alpha$.
\item
If $f(z) \in A$, then $f^2(z) \in Q_0$ and  $z$ satisfies one of the following properties:
\begin{enumerate}[b.i)]
\item
There exists an integers $N$ such that for all $n \geq N,\; f^{n}(z) \in Q_0$. In this case $\lim f^{n}(z)= \alpha$.
\item
For all integers $n,\; f^{2n}(z) \in Q_0$ and $f^{2n+1}(z) \in A$. In this case $\lim f^{n}(z)= \theta$.
\end{enumerate}
\end{enumerate}

\end{prop}

\begin{proof}
Let $z= (x_0, y_0) \in Q_0= [0,1+c] \times [0, a_2] $. If $f(z)\in Q_0$, then $z \in [0, a_2] \times [0, a_2]$.
Hence  $f^n (z) \in   [0, a_2] \times [0, a_2] \subset Q_0$ for all integers $n \geq 0$.

\vspace{0.5em}

{\bf Claim}: ${\lim f^{n}(z)= \alpha= (a_1, a_1)}$.

\vspace{0.5em}

{\it Case 1}: ${\min \{x_0, y_0\} < a_1 \textrm{ and } \max \{x_0, y_0\}  \leq a_1}$.

Assume first that $x_0 < a_1$ and $y_0 <a_1$. Observe that $x_1 < a_1 ^2+ c= a_1$ and $y_1= x_0 <a_1$.
Hence $f^{n}(z) \in [0, a_1[ \times [0, a_1[$ for all $n \geq 0$.
There exists a positive real number {\bf $0 \leq d_0 < 1$ } such that $\min\{x_0,y_0\} \geq d_0 a_1$.
Thus $x_1= x_0 y_0+ c \geq d_1 a_1$ where $d_1= g(d_0)$ and $g(x)= a_1 (x^2 -1) +1$.
Since $d_0 < 1$, we deduce that $d_0 < d_1$.

On the other hand, we have $y_2 = x_1 \geq d_1 a_1$ and $x_2 = x_1 x_0+ c \geq d_1 a_1.$
Hence we deduce by induction that $f^{2n}(x_0, y_0)= (x_{2n}, y_{2n})$ satisfies $\min\{x_{2n},y_{2n}\} \geq d_n a_1$, where $d_n$ is an increasing sequence satisfying $d_n = g(d_{n-1})$ for all integers $n \geq 1$.
Thus $\lim d_n \in \{1, \frac{1}{a_1} -1\}$. Since $d_n <1$ for all $n$ and
$ \frac{1}{a_1} -1 = \frac{ a_2 }{a_1} >1$,
we deduce that $\lim d_n= 1$.
Thus  $f^{2n}(x_0, y_0)$ converges to $\alpha$, and hence  $f^{n}(x_0, y_0)$ converges to $\alpha$.

Now if $x_0 < a_1,\; y_0 = a_1$ , then $x_1= x_0 <a_1$ and $y_1= x_0 <a_1$ and we are done.
If $x_0 = a_1,\; y_0  < a_1$ , then $x_2 <a_1$ and $y_2= x_1 <a_1$ and we are also done.

\vspace{0.5em}

{\it Case 2}: ${\min\{x_0,y_0\}  \geq a_1 \textrm{ and }\max (x_0, y_0)  > a_1}$.

As in Case 1, we can assume
 that $x_0 > a_1$ and $y_0 > a_1$.  Hence  $f^{n}(z) \in ]a_1, +\infty[ \times [a_1, +\infty[$ for all $n \geq 0$.
There exists a positive real number $e_0 > 1$ such  that $\max (x_0, y_0) \leq e_0 a_1$ where $1< e_0 < \frac{a_{2}}{a_{1}}.$
Thus $x_1= x_0 y_0+ c \leq e_1 a_1$ where $e_1= g(e_0)$.
Since $1 < e_0 < a_2 / a_1$, it is easy to see that $e_1 < e_0$.
We deduce as in case 1) that $\max \{x_{2n}, y_{2n}\} \leq e_n a_1$, where $e_n$ is a decreasing sequence satisfying $e_n = g (e_{n-1})$ for all integers $n \geq 1$ with $e_0 >1$.
Thus $\lim e_n \in \{1, \frac{1}{a_1} -1\}$. Since, for all integers $n \geq 0,\;  e_n  \leq e_0 < \frac{a_{2}}{a_{1}}= \frac{1}{a_1}-1$, we deduce that $\lim e_n= 1$.
Hence  $f^{2n}(x_0, y_0)$ converges to $\alpha$, and therefore  $f^{n}(x_0, y_0)$ converges to $\alpha$.

\medskip
{\it Case 3}: ${x_0 \geq a_1,\; y_0 \leq a_1 \textrm{ and }(x_0, y_0) \ne (a_1, a_1)}$.

Assume that $x_0 > a_1$ and $y_0 < a_1$.
If $x_1 \geq a_1$, then $\min\{x_1,y_1\} = \min\{x_1, x_0\} \geq a_1 $ and $\max\{x_1, y_1\}>a_1$.
Then by Case 2, we are done.

Now, suppose that $x_1= x_0 y_0 +c < a_1$.
We can also assume that $x_2= x_0 x_1+ c >a_1$ and $x_3= x_1 x_2+ c < a_1,$ otherwise,
if $x_2 \leq a_1$, then we are done as in Case 1.
If $x_2 > a_1$ and $x_3 \geq a_1$, then we are done  as in Case 2.

On the other hand, since $x_1= x_0 y_0 +c < a_1 <x_2= x_0 x_1+ c$, we deduce that
$y_0 < x_1= y_2$.
Since
$x_3= x_1 x_2 +c < a_1 <x_2= x_0 x_1+ c$, we obtain that
$x_2 < x_0$.
Then, we deduce by induction that $(x_{2n})$ is decreasing and $(y_{2n})$ is increasing.
Let $l= \lim x_{2n}$ and $l'= \lim y_{2n}$.
Then
 \begin{eqnarray}
\label{l0}
 x_{2n+2}= x_{2n} (x_{2n} y_ {2n}+ c)+ c  \mbox { and } y_{2n+2}= x_{2n} y_{2n}+ c >1.
 \end{eqnarray}
  We deduce by (\ref{l0}), that
$l=l' \in \{a_1, a_2\}$ . Since $y_{2n} <a_1$ for all integers $n \geq 0$, then $l=a_1$. Hence
 $f^{n}(x_0, y_0)$ converges to $\alpha$.

\medskip
{\it Case 4}: ${x_0 \leq a_1,\; y_0 \geq a_1 \textrm{  and }(x_0, y_0) \ne (a_1, a_1)}$.

Suppose that $x_0 < a_1$ and $y_0 > a_1$.
Then $y_1= x_0 < a_1$  and $ x_1 < a_1$ or $x_1 \geq a_1$.
In both cases, we are done because of Cases 1 or 3.
Hence we obtain the claim and Proposition \ref{prop:r0a}~a).

\vspace{0.5em}

Now, suppose that $z= (x, y) \in Q_0= [0, 1+c] \times [0, a_2]$ and $f(z)= (xy+c, x) \in A= [0, a_2] \times [a_2, +\infty[$, then $f(z) \in [0, a_2] \times [a_2, 1+c[$, hence 
\[
f^2(z) \in [c, a_2 (1+c)+c] \times [0, a_2] \subset [c, 1+c] \times [0, a_2] \subset Q_0.
\]
Assume that for all integers $n,\; f^{2n}(z)= (x_{2n}, x_{2n-1}) \in Q_0 = [0, 1+c] \times [0, a_2]$ and $f^{2n+1}(z)= (x_{2n+1}, x_{2n}) \in A= [0, a_2] \times [a_2, +\infty]$.
Then for all integers $n \geq 0,\;  x_{2n} \in [a_2, 1+c]$ and
$ x_{2n+1} \in [0, a_2]$. Hence for all integers $n \geq 2,\;   x_{2n-2}x_{2n-1}+c =  x_{2n} \geq x_{2n+1}= x_{2n}x_{2n-1}+c $.
 Thus $ x_{2n-2} \geq x_{2n}$.
 On the other $ x_{2n}x_{2n-1}+c= x_{2n+1} \leq  x_{2n+2} =  x_{2n}x_{2n+1}+c$.
 Hence $ x_{2n-1} \leq x_{2n+1}$. Thus $ y_{2n}= x_{2n-1} \leq y_{2n+2}=x_{2n+1} $, for all integers $n \geq 0$.
 Let $l= \lim x_{2n}$ and $l'=  \lim y_{2n}$.
 Then $l= l' \in \{a_1, a_2\}$.
 Since $x_{2n} \geq a_2$ for all integers $n \geq 0$, we deduce that $l=a_2$.
 Hence $f^{n}(x_0, y_0)$ converges to $\theta$.

Now, if $f^{2n}(z)$ and $f^{2n+1}(z)$ are both   in $Q_0$ for some integer $n= n_0$,
 then $f^{k} (z) \in Q_0$ for all integers $k \geq 2 n_0$. Hence by Proposition \ref{prop:r0a}~a), 
 $f^{n}(z)$ converges to $\alpha$.
\end{proof}

 \begin{obs}
 \label{ro}
If $z \in Q_0 $ such that $f(z) \in Q_0$, then $f^{n}(z) \in Q_0$ for all integers $n \geq 2$ and $\lim f^{n}(z)= \alpha$.
 \end{obs}

We continue our analysis, concentrating now on the points in $Q_3$. Again, we go through a list of various cases, showing that in most situations the coordinates of the points are monotone under iteration of an appropriate iterate of the map $f_c$.

 \begin{prop}
\label{r1}
Let $z \in Q_3 \cap K^{+} \setminus \{(1+c, -1)\} $,
then
\begin{enumerate}[a)]
\item
 If $f(z) \in Q_0$, then  $f^2(z) \in Q_0$ and hence $\lim f^{n}(z)= \alpha$.
\item
If $f(z) \in A$, then  $f^2(z) \in Q_0$ and hence
$\lim f^{n}(z)= \alpha$ or $\theta$.
\item
 If $f(z) \in Q_1$, then there exists an integer $N \geq 0$ such that $f^n(z)  \in Q_0 \cup A$ for all  $n \geq N$, and hence
$\lim f^{n}(z)= \alpha$ or $\theta.$
\end{enumerate}

\end{prop}

\begin{proof}
Let $z= (x_0, y_0) \in [0, 1+c] \times [-1, 0]=Q_3.$

a) If $f(z)= (x_0y_0+c, x_0) \in Q_0=  [0, 1+c] \times [0, a_2]$, since $x_0 y_0 \leq 0$, we deduce that  $f(z) \in [0, c] \times [0, a_2]$. Thus $f^2(z) \in [c, c a_2+ c] \times [0, c] \subset [c, 1+c] \times [0, c] \subset Q_0$, since $c a_2 \leq 1$.
By Proposition \ref{prop:r0a}, we deduce that  $\lim f^{n}(z)= \alpha$.

b) Now, assume that  $f(z) \in A= [0, a_2] \times [a_2, + \infty[$. Then  $f(z) \in [0, c] \times [a_2, 1+c]
$. Hence $f^2(z) \in [c, c(1+c)+ c] \times [0, c] \subset [c, 1+c] \times [0, c] \subset Q_0$, since $c(1+c) \leq a_2 (1+c) <1$.
Then by Proposition \ref{prop:r0a}, we obtain 2).

c) Assume without loss of generality that $z= (x_0, y_0) \in [0, 1+c] \times [-1, 0]=Q_3$ and  $ f^{n}(z) \not \in Q_0 \cup A$ for all integers $n \geq 1$.
Then
  $f^2(z) \in Q_3$ or   $f^3(z) \in Q_3$.

\medskip
 {\it Case 1}: ${f^2(z) \in Q_3}$.

 In this case $f(z) \in Q_1 = [-1, 0] \times [0, 1+c]$ and $f^{3}(z) \in Q_1$.

 {\it Claim 1}: ${x_0 \geq x_2 \textrm{ and }y_0 \leq y_2}$.

\vspace{0.5em}

 Indeed, we have
\[
-1 \leq x_1= x_0y_0+c \leq 0,\; 0 \leq  x_{2}= x_1 x_0 +c \leq 1+c
\]
 and
$-1 \leq x_3= x_2 x_1+c \leq 0$.
Then $x_1 \leq x_2$ and hence $y_0 \leq x_1=y_2$.
On the other hand $x_3 \leq x_2$ implies that $x_2 \leq x_0$.

\vspace{0.5em}

{\it Case 2}: $f^3(z) \in   Q_3$.

 In this case $f(z) \in Q_1 $ and $f^{2}(z) \in Q_2$.

{\it Claim 2}: $x_0 \geq x_3$ and $y_0 \leq y_3$.

\vspace{0.5em}

 Indeed, assume that $y_0  \geq y_2 = x_0y_0+c$. Then $-c \geq y_0(x_0-1).$
 Since $y_0 \leq 0$, we deduce that $x_0 \geq 1$. Hence $0 \leq x_0-1 \leq c.$
 Since $-1 \leq y_0 \leq 0$, then $y_0(x_0-1) \geq -c$. Absurd.
Hence $y_0  \leq y_2 = x_0y_0+c$. Thus $x_0y_0+ c \leq
 x_0(x_0 y +c)+c = y_3.$ Then $y_2 \leq y_3$ and
hence $y_0 \leq y_3$. On the other hand, since $x_1= x_0y_0+c \leq 0$,
 then  $y_0 (x_0y_0+c) \geq y_3 (x_0y_0+c)$.
 Therefore
 $$y_0 (x_0y_0+c)+ c \geq ((x_0y_0 +c)x_0 +c) (x_0y_0+c)+ c = x_3.$$
 To have $x_0 \geq x_3$, it is enough to prove that
 $ x_0 \geq y_0 (x_0y_0+c)+ c$, which is equivalent to
 $ (x_0y_0+c -x_0)(y_0+1) \leq 0$. This is true since $ y_0 \geq -1,\; x_0y_0+c \leq 0$ and $x_0  \geq 0$.
Then we obtain the claim.

Therefore, for any $(x, y) \in [0, 1+c] \times [-1, 0]=Q_3$, there exists an increasing sequence $(n_k)_{k \geq 0}$ such that $f^{n_{k}}(x, y)= (x_{n_k}, y_{n_k})$ satisfies
$n_0= 0, n_{k+1}- n_k \in \{2,3\},\; x_{n_{k}} \geq  x_{n_{k+1}},\; y_{n_{k}} \leq  y_{n_{k+1}}$.
Therefore  $f^{n_{k}}(x, y)$ converges to $(a, b) \in [0, 1+c] \times [-1, 0]$. Observe that $a < 1+c$ and $b >-1$.
Writing for all integers $k \geq 0,\; x_{n_{k}}= a+ \varepsilon_k$ and  $y_{n_{k}}= b- \delta_k$ where $ \varepsilon_k$ and  $ \delta_k$ are decreasing sequences of non-negative real numbers converging to $0$, we obtain that
$$\delta_{k+1}= b - (a+ \varepsilon_k) (b - \delta_k)- c,\;  \varepsilon_{k+1}= -a + (b - \delta_{k+1}) (a+  \varepsilon_{k})+ c,\; \mbox { if } n_{k+1}-n_k= 2,$$
or
$$\delta_{k+1}= b - ((a+ \varepsilon_k) (b - \delta_k) + c) (a+  \varepsilon_{k})- c,\;  \varepsilon_{k+1}= -a + (b - \delta_{k+1}) ((a+  \varepsilon_{k})(b - \delta_{k})+ c) + c,$$
if $ n_{k+1}-n_k= 3$.
Letting  $k$ tend to infinity, we obtain in the first case that
$a=b$ and $ab+c =a$. Hence $a \in \{a_1, a_2\} \subset [0, 1+c[ \times ]-1, 0]  $, which is absurd.

In the second case, we obtain $a=b$ or $ab+c+1 =0$.
If $a=b$, we obtain a contradiction as above.
If $ab+c+1 =0$, we also have a contradiction, since $(a, b) \in [0, 1+c[ \times ]-1, 0].$
\end{proof}

\subsubsection {Dynamics of $f$ outside ${\cal D} \cap  K^{+}$}

In this part, we rule out the existence of other basins of attractions inside $K^+$. That is, we show that points that are in $\mathcal D \cap K^+$ and outside of the basin of the attracting fixed point must be on 
its boundary. At the same time we prove that the boundary of this basin of attraction is simply made of two stable manifolds: the stable manifolds of the 3-cycle $\{p,f_c(p),f^2_c(p)\}$ and the stable manifold of the fixed point $\theta$.

We will study the itineraries of the points along the following partition into rectangular regions:

\begin{minipage}{0.5\textwidth}
\begin{figure}[H]
\includegraphics[scale=0.26]{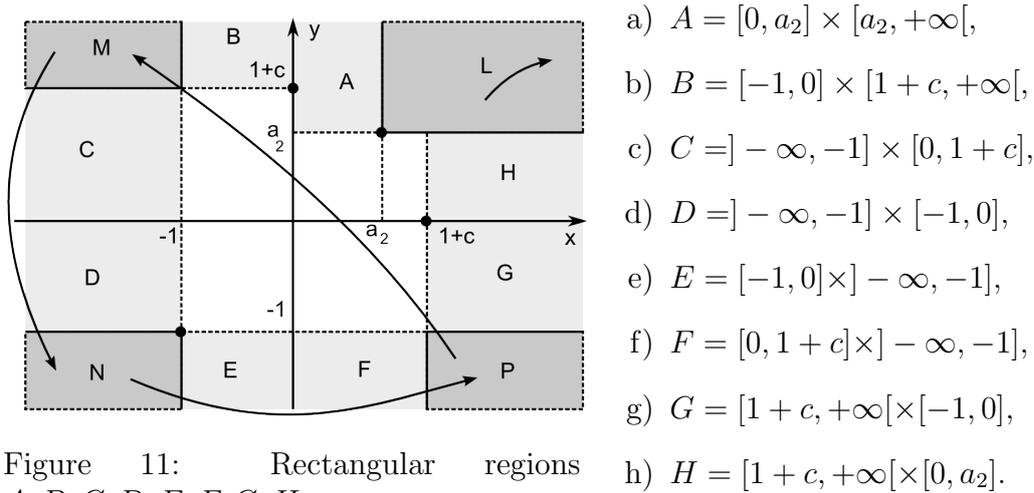}
\caption{Rectangular regions $A,B,C,D,E,F,G,H$}
\label{fig:ABCD}
\end{figure}
\end{minipage} \hfill
\begin{minipage}{0.45\textwidth}
 \begin{enumerate}
\item[a)] $A= [0,a_2] \times [a_2, +\infty[,$
\item[b)] $ B= [-1,0] \times [1+c, +\infty[, $
\item[c)] $C= ]-\infty, -1] \times [0, 1+c],$
\item[d)]  $D= ]-\infty, -1] \times [-1, 0],$
\item[e)] $E= [-1,0] \times ]-\infty, -1],$
\item[f)] $F= [0,1+c] \times ]-\infty, -1],$
\item[g)] $G= [1+c,+\infty[ \times [-1, 0],$
\item[h)] $ H= [1+c,+\infty[ \times [0, a_2]$.
\end{enumerate}

\end{minipage}

\vspace{0.3cm}

One gets a very useful crude description of the dynamics by describing how the various subsets of the partition are mapped into each other:

\begin{lem}
\label{prop:outsid}
The induced dynamics on the sets $A, B, \ldots, L$ is as follows:
\begin{enumerate}[a)]
\item
$f(A) \subset Q_0 \cup H,\; f(B ) \subset D \cup Q_2 \cup Q_3 $;
\item
$f(C ) \subset N \cup E \cup F,\;  f(D ) \subset P \cup F$;
\item
$f(E ) \subset Q_3 \cup G,\;  f(F ) \subset Q_0 \cup Q_1 \cup C$;
\item
$f(G ) \subset A \cup B \cup M,\;  f(H ) \subset A \cup L$.
\end{enumerate}
\end{lem}

\begin{proof} 
All these inclusions are easy to check:
\begin{enumerate}[i)]
\item $
f(B ) \subset (-\infty, c] \times [-1, 0] \subset D \cup Q_2 \cup Q_3$.
\item
 $f(C) \subset (-\infty, c] \times [-\infty, -1] \subset N \cup E \cup F,$
 
 \item $f(D ) \subset [c, +\infty] \times [-\infty, -1] \subset F \cup P$.

\item $ f(E ) \subset [c, +\infty] \times [-1, -0] \subset G \cup Q_3.$

\item  $ f(G ) \subset (-\infty, c] \times [1+c, +\infty] \subset A \cup B \cup M$.

\item $ f(F ) \subset (-\infty, c] \times [0, 1+c] \subset Q_0 \cup Q_1 \cup A \cup C$.
\end{enumerate}

On the other hand, $f(x, y)= (xy+c, x) \in A= [0, a_2] \times [a_2, +\infty[$ implies that
$x \geq a_2$ and $-c \leq xy$. Since $c < a_2$, we deduce that $y >-1$. Hence $(x,y)$ cannot belong to
$F= [0, 1+c] \times ]-\infty, -1]$.
\end{proof}

The next proposition is the most technical. We analyze the induced dynamics on the partition (which is best represented by its transition graph, see figure \ref{fig:Cycles}). We identify then certain cycles of length 6, along which coordinates of iterates become monotone sequences. This simple fact is then used to show the convergence towards some attracting cycles.

\par\smallskip\noindent
\begin{figure}[h]
\centering
\begin{minipage}[h]{.8\textwidth}

\begin{tikzpicture}[->,>=stealth',shorten >=1pt,auto,node distance=2.8cm,
                    semithick]
  \tikzstyle{every state}=[fill=none,draw=black,text=black]
  \node[state] (D')                    {$D'$};
  \node[state]         (B') [below right of=D'] {$B'$};
   \node[state]         (F') [above right of=D'] {$F'$};
  \node[state]         (C') [right of=F'] {$C'$};
  \node[state]         (E') [right of=C'] {$E'$};
  \node[state]         (G') [above right of=E'] {$G'$};
  \node[state]         (H') [above of=C'] {$H'$};
  \node[state]         (A') [left of=H'] {$A'$};
  \path (B') edge           (D');
  \path (D') edge           (F');
  \path (F') edge   [bend left]        (C');
  \path (C') edge           (E');
  \path (E') edge           (G');
    \path    (C') edge [bend left ]  (F');
\path	(A') edge   [bend left]  (H');
\path	(H') edge [bend left] (A');
\path	(G') edge [bend left] (A');
\path	(G') edge [bend left] (B');
\end{tikzpicture}
\captionof{figure}{Cycles in the partition}
\label{fig:Cycles}
\end{minipage}
\end{figure}

\begin{prop}
\label{prop:r0}
If $z \in  K^{+} \setminus \bigcup_{n=0}^{+\infty} f^{-n} ({\cal D} \cap  K^{+})$ , then $z \in W^{s}(p) \cup
W^{s}(f(p)) \cup W^{s}(f^2(p))$.

\end{prop}

\begin{proof}
Let $V= K^{+} \setminus \bigcup_{n=0}^{+\infty} f^{-n} ({\cal D} \cap  K^{+})$ and  suppose that $V \neq \emptyset$.
Then $ V \subset A \cup B \cup C \cup D \cup E \cup F \cup G \cup H$.
For all $X \in \{A, B, C, D, E, F, G, H\}$, denote by $X'= X \cap V$.
By Lemma \ref{prop:outsid}, we have (see Figure~\ref{fig:Cycles})
$$f(A') \subset H',\; f(H') \subset A'$$
$f(B') \subset D',\; f(D') \subset F',\; f(F') \subset  C',\; f(C') \subset  F' \cup E',\; f(E') \subset G' ,\; f(G') \subset A' \cup B'$.

Observe that in Figure~\ref{fig:Cycles} we have three types of cycles: two cycles of length two $A', H'$ 
and $F', C'$
and cycles of length more than $6$ beginning in one of the vertices  $F', C', E', G', B', D'$.
Let $z= (x_0, y_0) \in V$. We will prove that the orbit of $z$ cannot stay in a cycle.

\vspace{0.5em}

{\it Case 1}:  $z= (x_0, y_0) \in A'$.

Then for all integers $n \geq 1,\;  f^{2n} (x_0, y_0)= (x_{2n}, y_{2n})= (x_{2n}, x_{2n-1}) \in A$ and $f^{2n-1} (x_0, y_0)= (x_{2n-1}, x_{2n-2}) \in H$.
Hence
$$ 0 \leq x_0 \leq a_2,\; a_2 \leq y_0 <+\infty,$$
\begin{eqnarray}
\label{b0}
 1 \leq x_0 y_0 < +\infty,\; -c \leq x_0 x_1 \leq a_2-c,\; 1 \leq x_1 x_2 < + \infty.
 \end{eqnarray}
By (\ref{b0}) and the fact that $a_2 - c<1$, we deduce that $x_0 x_1 < x_0 y_0$ and $x_0 x_1 < x_2 x_1$.  Hence
 $$y_2= x_1 < y_0 \mbox {  and } x_0 < x_2.$$
 We deduce by induction that
 $(x_{2n})_{n \geq 0}$ is a increasing convergent sequence and $ (y_{2n})_{n \geq 0}$ is a decreasing convergent sequence. Let $l = \lim x_{2n}$ and
 $l'= \lim y_{2n}$.
 Since
 \begin{eqnarray}
\label{l00}
 x_{2n+2}= x_{2n} (x_{2n} y_ {2n}+ c)+ c  \mbox { and } y_{2n+2}= x_{2n} y_{2n}+ c >1,
 \end{eqnarray}
 we deduce by (\ref{l00}), that
$l=l' \in \{a_1, a_2\}$. Since $(l,l') \in A$, we obtain that $l=a_2$.
Since $(a_2, a_2)$ is a fixed point for $f$, then $\lim (x_{2n+1}, y_{2n+1})=(a_2, a_2) \in H= [1+c, +\infty[ \times [0, a_2]$, which is absurd.
Hence $A' = \emptyset$.

\vspace{0.5em}

{\it Case 2}:  For all integers $n \geq 0,\; f^{2n}(z) \in C' $ and $f^{2n+1}(z) \in F'$.

Then
$$ 0 \leq x_0 \leq -1,\; -1 \leq y_0  \leq 1+c,$$
\begin{eqnarray}
\label{c0}
 -c \leq x_0 y_0 \leq 1,\; -\infty < x_0 x_1 \leq -1-c,\; -c \leq x_1 x_2 \leq 1,
 \end{eqnarray}
By (\ref{c0}), we deduce $x_0 x_1 < x_0 y_0$ and $x_0 x_1 < x_2 x_1$. Hence
 $$y_2= x_1 > y_0 \mbox {  and } x_0 < x_2.$$
 We deduce by induction that
 $(x_{2n})_{n \geq 0}$ and and $ (y_{2n})_{n \geq 0}$ are increasing convergent sequences. If $l = \lim x_{2n}$ and
 $l'= \lim y_{2n}$, then by (\ref{l00}), we deduce that
$l=l' \in \{a_1, a_2\}$. This is absurd, because  $(l,l) \in C $.
Thus an infinite cycle of type $C \to F \to C$ cannot happen.

\vspace{0.5em}

{\it Case 3}:  For all integers $n \geq 0,\; f^{6n}(z) \in F,\;f^{6n+1}(z) \in C,\; f^{6n+2}(z) \in E,\; f^{6n+3}(z) \in G,\; f^{6n+4}(z) \in B, \; f^{6n+5}(z) \in D$.
Then
\begin{eqnarray}
\label{xx}
 0 \leq x_0 \leq 1+c,\; -\infty < y_0 \leq -1,
 \end{eqnarray}
\begin{eqnarray}
\label{x0}
 -\infty < x_0 y_0 \leq -1-c,\; -1-c \leq x_0 x_1 \leq -c,
 \end{eqnarray}
 \begin{eqnarray}
\label{x1}
-1-c \leq x_1 x_0 \leq -c,\; 1 \leq x_1 x_2 < + \infty,
 \end{eqnarray}
 \begin{eqnarray}
\label{x2}
-1-c \leq x_3 x_2 \leq -c,\; -\infty \leq x_3 x_4 \leq -1-c,
 \end{eqnarray}
\begin{eqnarray}
\label{x3}
-c \leq x_5 x_4 \leq 1,\; -\infty < x_5 x_6 \leq -1-c.
 \end{eqnarray}
By (\ref{x0}), we obtain $x_0 y_0 \leq x_0 x_1$, hence $y_0 \leq x_1= y_2$.
By (\ref{x2}), we obtain $x_4 \leq x_2$.
By (\ref{x3}) and (\ref{x1}) , we have $x_4 x_5 \leq x_1 x_2$. Hence $x_1 \leq x_5$.
Thus $y_0 \leq x_1 \leq x_5= y_6$.
By (\ref{x3}) and (\ref{x1}) , we have $x_5 x_6 \leq x_0 x_1$. Since $x_1 \leq x_5 \leq 0$ and $x_0 \geq 0$ and $x_6 \geq 0$, then we have $x_0 \leq x_6$.
Thus the sequences  $(x_{6n})_{n \geq 0}$ and $(y_{6n})_{n \geq 0}$ are increasing.

   \vspace{0.5em}

   {\bf Claim: $\lim f^{6n}(x_0, y_0)= (1+c, -1)= f(p)$.}

   Indeed, let us write $\lim x_{6n}= l$ and $\lim y_{6n}= l'$.
   Then $f^{6}(l,l')= (l, l')$.
   Assume that $(l, l') \neq (1+c, -1)$ and put $f^6(l, l')= (a, b)$.
  We deduce as earlier that $l \leq a$ and $l' \leq b$. Since $0 \leq l < 1+c$ and $-\infty < l' < -1$, we have that $l < a$ or $l' < b.$ Thus  $f^{6}(l,l') \neq (l, l')$, which is absurd. This proves the claim.

   Let $z= (x_0, y_0) \in V$ and assume without loss of generality that $z \in F$,  then  there exists an increasing sequence $(n_k)_{k \geq 0}$ such that $f^{n_{k}}(x, y)= (x_{n_k}, y_{n_k})$ satisfies
   \[
   n_0= 0,\;  n_{k+1}- n_k \in \{2,6\},\; x_{n_{k}} \leq  x_{n_{k+1}},\; y_{n_{k}} \leq  y_{n_{k+1}}.
   \]
Moreover, there exists an integer $N\geq 0$ such that $ n_{k+1}- n_k =6$ for all $k \geq N$, otherwise one would encounter an infinite cycle $F \to C \to F$. In this case, there exists a subsequence $f^{n_{k_{i}}}(z)$ of $f^{n}(z)$ converging   to the fixed point $(a_2, a_2)$  with  $f^{n_{k_{i}}+1}(z)$ in $C= [-1, 0] \times [-1, 0]$, which is absurd.
Hence  $ n_{k+1}- n_k =6$ for all $k \geq N$ and we deduce that $\lim f^{6n}(x_0, y_0)= (1+c, -1)= f(p)$.
\end{proof}

\subsubsection { Description of $K^{-}$}
This section gives a detailed description of $K^-(f_c)$ following the line of study of the set $K^+$. We prove that in the parameter range $0<c<1/4$ the set $K^-$ is simply made of the union of two unstable manifolds of periodic cycles.

 \begin{teo}
If  $0< c <1/4$, then $K^{-}= W^{u}(\theta)\cup  W^{u}(p)\cup W^{u}(f(p)) \cup W^{u}(f^2(p))$
 where $p= (-1, -1)$.
\end{teo}

Let $Z'= A \cup B\cup C\cup D\cup E\cup F\cup G \cup H.$
Observe that $\mathbb{R}^2 \setminus \textrm{int} ({\cal D} \cup L \cup M \cup N \cup P)= Z'$.
Let us define $Z= Z' \setminus \{\alpha, \theta, p, f(p), f^2(p)\}$.
The following proposition describes the induced dynamics on the partition of $Z'$
\begin{prop}
\label{prop:creal}
The following properties are valid:
\begin{enumerate}[a)]
\item
$f^{-1}(A) \subset H \cup G \cup Q_0 \cup Q_3,\; f^{-1}(H) \subset A ;$

 $f^{-1}(G) \subset E,\; f^{-1}(E) \subset C,\;  f^{-1}(C) \subset F;$

$f^{-1}(F) \subset C \cup D,\;  f^{-1}(D) \subset B,\; f^{-1}(B) \subset G$;

$ f^{-1}(Q_0) \subset Q_0 \cup Q_3 \cup A \cup F,\;
f^{-1}(Q_1 ) \subset  Q_3 \cup F,\;
f^{-1}(Q_2 ) \subset Q_1 \cup B$;

$f^{-1}(Q_3 ) \subset Q_1 \cup Q_2 \cup B \cup E$.
\item
For all $(x, y) \in Z,\;  \| f^{-n} (x, y) \|$ diverges to $ \infty$ as $n$ goes to $+\infty$.
\item
$\mathbb{R}^2 \setminus K^{-}= \bigcup_{n=0}^{+\infty} f^{n}(Z) $, where $ f^{n}(Z) \subset f^{n+1}(Z) $ for all integers $n \geq 0$.
\end{enumerate}
\end{prop}

\begin{proof}
a) is easy to prove and can be deduced immediately  from Lemmas \ref{insid} and \ref{prop:outsid}.

b) Let $z=(x_0, y_0) \in Z'$ and assume first that $z  \in F \setminus \{(1+c, -1)\}$.

\vspace{0.5em}

{\it Case 1:} $ f^{ -n}(z) \not \in C$ for all integers $n \geq 1$.

 Then by Proposition \ref{prop:creal}~a), we deduce that  for all integers $n \in \mathbb{N}$
  \[
 f^{ -6n}(z)= (x_{-6n}, y_{-6n}) \in F,\; f^{-1- 6n}(z) \in D,\;
f^{-2- 6n}(z) \in B,
\]
and also
\[
f^{-3- 6n}(z) \in G,\; f^{-4- 6n}(z) \in E,\; f^{-5- 6n}(z) \in C.
\]
Hence by Case 3 of the proof of Proposition \ref{prop:r0}, we deduce that $( x_{-6n})_{n \geq 0}$ and $(y_{-6n})_{n \geq 0}$ are strictly decreasing (since $z  \neq (1+c, -1)$). Let $l=  \lim x_{-6n}$ and $l'= \lim y_{-6n}$. Assume that $(l, l') \in \mathbb{R}^2$. Hence $f^{-6}(l, l')= (l, l')$. Thus as done in the proof of Case 3 of Proposition \ref{prop:r0}, we deduce $ (l, l')= (1+c, -1)$, which is absurd since $z \neq (1+c, -1)$. Therefore   $\lim x_{-6n}= 0$ and $\lim y_{-6n}= - \infty$.
Hence 
\[
\lim f^{-6n}(z)= (0, -\infty),\; \lim f^{-6n-1}(z)= (-\infty,0),\; \lim f^{-6n-2}(z)= (0, +\infty) ,
\]
 and
 \[
  \lim f^{-6n-3}(z)= (+\infty, 0) ,\; \lim f^{-6n-4}(z)= (0, -\infty).
 \]
 Also
$ \lim f^{-6n-5}(z)= (-\infty,0).$

\vspace{0.5em}

{\it Case 2:} There exists an integer $N \in \mathbb{N}$ such that $ f^{ -N}(z) \in C$.

In this case $N= -6n-1$ where $n \geq 0$ is an integer, then one obtains
\[
f^{-6n}(z) \in F,\; f^{-6n-1}(z) \in C \textrm{ and }f^{-6n-2}(z) \in F.
\]
Using the same method as in Case 2 in the proof of Proposition \ref{prop:r0}, we deduce that $x_{-6n} > x_{-6n-2}$ and $y_{-6n} > y_{-6n-2}$.

Hence, if we suppose that there exists an integer $n_0$ such that for all $k \geq n_0$
\[
 f^{-k}(z) \in F,\; f^{-k-1}(z) \in C \textrm{ and }f^{-k-2}(z) \in F,
 \]
  we deduce that
$f^{-n_0- 2n}(z)$ converges to $(a, b) \in \overline{F}$.
If $(a, b) \in \mathbb{R}^2$, then $(a, b)$ is one of the fixed points of $f$. This is absurd.
Hence $b= -\infty$ and $a=0$.

Now, assume that for $z= (x_0, y_0) \in F$,  there exists an increasing sequence $p_k >0$ such that $f^{-p_{k}}(x, y)= (x_{-p_k}, y_{-p_k}) \in F$ satisfies
\[
p_0= 0,\; p_{k+1}- p_k \in \{2,6\},\; x_{-p_{k+1}} \leq  x_{-p_{k}},\; y_{-p_{k+1}} \leq  y_{-p_{k}}.
\]
We deduce that  $f^{-p_{k}}(x, y)$ converges to $(a, b)= (0, -\infty) \in \overline{F}.$
Hence, we deduce the result in the cases where $z \in D, B, G, E$ and $C$.

\vspace{0.5 em}

Now, assume that $z \in A= [0,a_2] \times [a_2, +\infty[ \setminus \{(a_2, a_2)\}$.
Suppose that for all integers $n \geq 0,\; f^{-2n}(z) \in A$ and $f^{-2n-1}(z) \in H$.
Then, we deduce as in Case 1 of the proof of Proposition \ref{prop:r0} that  the sequence $(x_{-2n})_{n \geq 0}$ is decreasing and the sequence  $(y_{-2n})_{n \geq 0}$ is increasing.
If the limit of $(x_{-2n}, y_{-2n})= (a, b)$ belongs to $\mathbb{R}^2$, we know that the limit is one of the fixed points of $f$. On the other hand, since $(x_{-2n}, y_{-2n}) \in A$, we deduce that $a < a_2$  and $b > a_2$, which is absurd. Hence  $\lim (x_{-2n}, y_{-2n})= (0, +\infty).$
If there exists an integer $n_0 \geq 1$ such that $f^{-n_0}(z) \in G$, then we are done.

We now assume the existence of an integer $n \geq n_0$ such that $f^{-n}(z)$ and $f^{-n-1}(z)$  are in $ Q_0$. Then because of remark \ref{ro}, we deduce that $z \in Q_0$. Moreover 
$z \in \operatorname{Int}(Q_0)$. But this cannot happen, as $z \in A$.
Thus, there exists an integer $k > n_0$ such that $f^{-k}(z) \in Q_3 \cup A \cup F$.

{\it Case 2.1}: If $f^{-k}(z) \in F$, then we are done.

\vspace{0.5 em}

{\it Case 2.2}: If $f^{-k}(z) \in A$.

Assume that $f^{-n}(z) \not \in Q_3$ for all integers $n >k$, then we can suppose (otherwise we are done) that $f^{-k-1- 2n}(z) \in Q_0$ and $f^{-k- 2- 2n}(z) \in A$   for all integers $n \geq 0$.
Let us write $f^{-k- 1- m}(z) = (x_m, y_m)$ for all integers $m \geq 0$. We deduce as in the proof of item 3 of Proposition \ref{prop:r0}, that  $(x_{2n})_{n \geq 0}$ is increasing and   $(y_{2n})_{n \geq 0}$ is decreasing.
  Since $(x_{2n}, y_{2n}) \in Q_0$ for all $n \geq 0$, we have that $(x_{2n}, y_{2n}) $ converges to $(a_1, a_1)$. Hence $f^{-1}(x_{2n}, y_{2n}) $ converges to $(a_1, a_1)$. That is absurd, since $f^{-1}(x_{2n}, y_{2n}) \in A $
  for all integers $n \geq 0$.

{\it Case 2.3}: If $f^{-k}(z) \in Q_3$, then $f^{-k-1}(z) \in Q_1 \cup Q_2 \cup E \cup B$.

If $f^{-k-1}(z) \in  E \cup B$, then we are done. We can then suppose that $f^{-n}(z) \in Q_1 \cup Q_2 \cup Q_3$ for all integers $n \geq k+1$, otherwise the orbit of $z$ under $f^{-1}$ will intersect $E \cup F \cup B$, and we are done.
Let us write
\[f^{-k}(z)= z'= (x_0, y_0) \in Q_3 =[0, 1+c] \times [-1, 0].
\]
Then
  $f^{-2}(z) \in Q_3$ or   $f^{-3}(z) \in Q_3$.
Suppose that  $f^{-2}(x_0, y_0)= (x_{-2}, y_{-2}) \in Q_3$, then as in Case 1 of the proof of 
Proposition~\ref{r1}~c),
we have $x_{-2} \geq x_0$ and $y_{-2} \leq y_0$.

Now, if $f^{-3}(x, y) \in   Q_3$, then as Case~2 of the proof of Proposition \ref{r1}~c),
we have
 $x_{-3} \geq x_0$ and $y_{-3} \leq y_0$.
Therefore, for any $z= (x_0, y_0) \in [0, 1+c] \times [-1, 0]=Q_3$, there exists an increasing sequence $p_k >0$ such that $f^{-p_{k}}(x, y)= (x_{-p_k}, y_{-p_k})$. This sequence satisfies
\[
p_0= 0, p_{k+1}- p_k \in \{2,3\},\; x_{-p_{k+1}} \geq  x_{-p_{k}},\; y_{-p_{k+1}} \leq  y_{-p_{k}}.
\]
Therefore  $f^{-p_{k}}(x, y)$ converges to $(a, b) \in Q_3= [0, 1+c] \times [-1, 0]$.
We now write for all integers $k \geq 0,\; x_{-p_{k}}= a- \varepsilon_k$ and  $y_{-p_{k}}= b+ \delta_k$ where $ \varepsilon_k$ and  $ \delta_k$ are decreasing sequences of non-negative real numbers converging to $0$.
Since $f^{ -p_{k}}(x, y)= f^{ p_{k+1}- p_k}(f^{ -p_{k+1}})(x, y)$, we
obtain that
$$b+ \delta_{k}=  (a- \varepsilon_{k+1}) (b + \delta_{k+1})+ c,\;  a -\varepsilon_{k}=  (b + \delta_{k}) (a -\varepsilon_{k+1})+ c,\; \mbox { if } p_{k+1}-p_k= 2,$$
or
$$
b+ \delta_{k}=  ((a+ \varepsilon_{k+1}) (b + \delta_{k+1}) + c) (a-  \varepsilon_{k+1})+ c,\;  a-\varepsilon_{k}= (b + \delta_{k}) ((a-  \varepsilon_{k+1})(b + \delta_{k+1})+ c) + c,$$
 if $ p_{k+1}-p_k= 3.$

\vspace{0.3cm}
{\bf Claim}: There exists an integer $N \geq 0$ such that for all $k \geq N,\; p_{k+1}- p_k= 3$.

Indeed, if the claim is not true, since $\varepsilon_k$ and  $ \delta_k$ converge to $0$, we deduce that $a=b= ab+c$. That is absurd, since $a >0$ and $b <0$.

\vspace{1em}

Letting $k$  tend to infinity in the second equation above,  we obtain that $a=b$ or $ab+c+1 =0$.
Since $a=b$ does not hold, we must have that $a= 1+c$ and $b= -1$. Hence $f^{-p_{k}}(x_0, y_0)$ converges to $(1+c, -1)= f(p)$.
Since  $f^{-p_{k}}(x_0, y_0) \in Q_3$ for all $k \geq 0$, then, when $n$ goes to infinity, $f^{n} \circ f^{-p_{k}}(x_0, y_0)$  converges to $\alpha= (a_1, a_1)$. That is absurd, since $(x_0, y_0) \in A= [0, a_2] \times [a_2, +\infty[$.
\end{proof}

\begin{lem}
\label{lemma:aa1}

$\left([0,a_1] \times [a_1, +\infty[) \cup ([a_1, +\infty[ \times [0, a_1]\right) \setminus \{a_1, a_1\} \subset \mathbb{C}^2 \setminus  K^{-}.$

\end{lem}

\begin{proof}
Let us write $A''= [0,a_1] \times [a_1, +\infty[$ and $H''= [a_1, +\infty[ \times [0, a_1] $.
As in the proof of Proposition \ref{prop:creal}~a), we can prove that
$$f^{-1}(A'') \subset H'' \cup G \cup Q_0 \cup Q_3,\; f^{-1}(H'') \subset A'' .$$
Repeating that proof, we obtain the result.
\end{proof}

\paragraph{Dynamics of $f^{-1}$ inside $ {\cal D } \cap  K^{-}$.}

This section gives a detailed description of the part of $K^-$ that stays in $\mathcal D$.
\vspace{0.5em}

As in our study of $K^+$ we start with a description of the dynamics induced on the partition by $f_c$.

\begin{lem}
\label{insidf-1}
The following results are valid:
\begin{enumerate}[a)]
\item
$f^{-1}(Q_0) \subset Q_0 \cup Q_3 \cup A \cup F$;
\item
$f^{-1}(Q_1 ) \subset  Q_3 \cup F$;
\item
$f^{-1}(Q_2 ) \subset Q_1 \cup B$;
\item
$f^{-1}(Q_3 ) \subset Q_1 \cup Q_2 \cup B \cup E$.
\end{enumerate}
\end{lem}

\begin{proof}
The proof is easy and is omitted.
\end{proof}

We now give the main result concerning the description of $K^-$ when the parameter~$c$ stays within the open interval $(0,1/4)$, namely that it consists of a union of unstable manifolds.
\begin{prop}
\label{dc}
$ {\cal D} \cap K^{-} \subset W^{u}(\theta) \cup W^{u}(p)\cup W^{u}(f(p)) \cup W^{u}(f^2(p))$.
More precisely,  $ Q_0 \cap K^- \subset  W^{u}(\theta),\;  Q_1 \cap K^- \subset  W^{u}(f^2(p)),\; Q_2 \cap K^- \subset  W^{u}(p)$ and $ Q_3 \cap K^- \subset  W^{u}(f(p)).$
\end{prop}

\begin{lem}
\label{r11}
Let $z = (x, y) \in \mathbb{R}^2$ be such that $f^{-n}(z) \in Q_0$ for all integers $n \geq 0$. 
Then  $\lim_{n\to\infty} f^{-n}(z)= \theta.$
\end{lem}

\begin{proof}
Let $z = (x, y) \in \mathbb{R}^2$ such that $f^{-n}(z) \in Q_0$ for all integers $n \geq 0$.
Let $d_0 <1$ such that $\min\{x,y\} \geq d_0 a_2 $.

{\bf Claim 1:} For all integers $n \geq 0,$ if one writes $f^{-2n} (x, y)= (x_{2n}, y_{2n})$ then we have 
$\min\{x_{2n},y_{2n}\} \geq d_n a_2 $
where $d_{n}= g(d_{n+1})$  with $g(x)= a_2 (x^2-1)+1$.

\vspace{0.5em}

Indeed: let us assume that the claim is true for $n$. We deduce that for all integers $n \geq 1$,
$$y_{2n-2}= x_{2n} y_{2n}+c \geq d_n^2 a_2^2+ c=  (d_n^2-1) a_2^2+ a_2= d_{n-1} a_2$$
and also
$$x_{2n-2}= (x_{2n} y_{2n}+c)x_{2n}+ c \geq d_{n+1}d_{n} a_2^2+ c>  d_n^2 a_2^2+ c= d_{n-1} a_2.$$
This concludes the proof of the first claim.

\vspace{0.5em}

{\bf Claim 2}: $(d_n)$ converges to $1$.

\vspace{0.5em}

We have
$g(x)-x = (x-1) ((x+1)a_2 -1)$.
If $1 / a_2 -1= a_1/ a_2 \leq x <  1$, then $g(x)< x$.

{\it Case 1}:  $\min\{x_{2n},y_{2n}\} \geq a_1$ for all integers $n$.
Then we can choose $ d_n \geq a_1/ a_2$ for all integers $n \geq 0$  and
in this case, we deduce that $d_n$ is increasing and $\lim d_n= 1$.

\vspace{0.5em}

{\it Case 2}: There exists an integer $k \geq 0$ such that $\min (x_{2k}, y_{2k}) < a_1$.

 Thus by lemma \ref{lemma:aa1}, $\max\{x_{2k},y_{2k}\} \leq a_1$.
 Hence $x_{2k+1}= y_{2k} \leq a_1$ which implies that $y_{2k+1} \leq a_1$, using lemma \ref{lemma:aa1}.

Thus,  for all integers $n \geq 2k$, we have $\max \{x_{n}, y_{n}\} \leq a_1$.

Writing $\min(x_{2k}, y_{2k}) = e_k a_1$, where $0 \leq e_k <1$,
we construct by induction a sequence of non-negative real numbers $(e_n)_{n \geq k}$ such that
$$\min\{x_{2n},y_{2n}\} \geq e_n a_1,\; \forall n \geq k,$$
where $e_{n}= (e_{n+1}^2 -1) a_1+ 1,\; \forall n \geq k$.
Hence $(e_n)_{n \geq k}$ is decreasing and converges to $l \in \{1, a_2/ a_1\}$. But this is absurd, since
$e_n < 1$ for all $n \geq k$. Hence, we obtain the claim. 

Thus $f^{-2n}(z)= (x_{2n}, y_{2n})$ converges to $\theta$ and it follows that $\lim f^{-n}(z)= \theta$.
\end{proof}

\medskip
\begin{proof}[Proof of Proposition \ref{dc}.]

Let $z= (x_0, y_0) \in Q_3 \cap K^-  \subset [0, 1+c] \times [-1, 0]$.
Then
  $f^{-2}(z) \in Q_3$ or   $f^{-3}(z) \in Q_3$.
Suppose that  $f^{-2}(x_0, y_0)= (x_{-2}, y_{-2}) \in Q_3$, then as case 1 of the proof of 
Proposition~\ref{prop:creal}~c),
we have $x_{-2} \geq x_0$ and $y_{-2} \leq y_0$.
Now, if $f^{-3}(x, y) \in   Q_3$, then as Case~2 of the proof of Proposition \ref{r1}~c),
we have
 $x_{-3} \geq x_0$ and $y_{-3} \leq y_0$.
Therefore, for any $z= (x_0, y_0) \in [0, 1+c] \times [-1, 0]=Q_3$, there exists an increasing sequence $p_k >0$ such that $f^{-p_{k}}(x, y)= (x_{-p_k}, y_{-p_k})$ satisfies
\[
p_0= 0, p_{k+1}- p_k \in \{2,3\},\; x_{-p_{k+1}} \geq  x_{-p_{k}},\; y_{-p_{k+1}} \leq  y_{-p_{k}}.
\]
Therefore  $f^{-p_{k}}(x, y)$ converges to $(a, b) \in Q_3= [0, 1+c] \times [-1, 0]$.
Let us write for all integers $k \geq 0,\; x_{-p_{k}}= a- \varepsilon_k$ and  $y_{-p_{k}}= b+ \delta_k$ where $ \varepsilon_k$ and  $ \delta_k$ are decreasing sequences of non-negative real numbers converging to $0$.
 We deduce, as before, that $a= 1+c$ and $b= -1$. Hence $f^{-p_{k}}(x_0, y_0)$ converges to $(1+c, -1)= f(p)$.
Thus $z=(x_0, y_0) \in   W^{u}(f(p)) $ and so $Q_3 \cap K^- \subset  W^{u}(f^2(p))$.

On the other hand, by Lemma \ref{insidf-1} and Proposition~\ref{prop:creal}~b), we deduce that $f^{-1}(Q_1 \cap K^- ) \subset  Q_3 \cap K^-$.
Hence
\[
Q_1 \cap K^- \subset  W^{u}(f^2(p)),\; Q_2 \cap K^- \subset  W^{u}(p).
\]
Let $z \in  Q_0 \cap K^-$.
If $f^{-n}(z) \in Q_0$ for all integers $n \geq 0$, then we are done by Lemma \ref{r11}.
Now assume that there exists an integer $k \geq 1$ such that $f^{-k}(z) \in Q_3 \cap K^-$,
then $\lim_{m \to +\infty} f^{-k- 3m}(z)= f(p)= (1+c, -1).$
\end{proof}

\paragraph{Dynamics of $f^{-1}$ inside $(\mathbb{R}^2 \setminus {\cal D}) \cap  K^{-}$.}
This part studies the behaviour of points of $K^-$ outside of $\mathcal D$.
\vspace{0.5 em}

\begin{lem}
\label{outsid}
The following results are valid:
\begin{enumerate}[a)]
\item
$f^{-1} (L)\subset Q_0 \cup H \cup L $;
\item
$f^{-1}(M ) \subset G \cup P$;
\item
$f^{-1}(N ) \subset C \cup M$;
\item
$f^{-1}(P )  \subset D \cup N$.
\end{enumerate}
\end{lem}

\begin{proof}

a) $f^{-1} (L)= f^{-1} ([a_2, +\infty[ \times [a_2, +\infty[) \subset [a_2, +\infty[ \times [0, +\infty[
\subset Q_0 \cup H \cup L.$

b)  $f^{-1} (M)= f^{-1} (]-\infty, -1] \times [1+c, +\infty[) \subset[1+c, +\infty[ \times ]-\infty, 0]
\subset G \cup P.$

c)  $f^{-1} (N)= f^{-1} (]-\infty, -1] \times ]-\infty, -1]) \subset ]-\infty, -1] \times [0, +\infty[
\subset C \cup M.$

d)  $f^{-1} (P)= f^{-1} ([1+c , +\infty] \times ]-\infty, -1]) \subset ]-\infty, -1] \times ]-\infty, 0[
\subset  D \cup N$.
\end{proof}

\begin{prop}
\label{r0}
Let $z \in  K^{-} \setminus \bigcup_{n=0}^{+\infty} f^{n} ({\cal D} \cap  K^{-})$. Then $z \in W^{u}(\theta) \cup W^{u}(p) \cup
W^{s}(f(p)) \cup W^{s}(f^2(p))$.
\end{prop}
\begin{proof}
Let  $z= (x, y) \in  K^{-} \setminus \bigcup_{n=0}^{+\infty} f^{n} ({\cal D} \cap  K^{-})$.
then $z \in L \cup M \cup N \cup P$.

{\it Case 1: $z \in L$.}

Then by Proposition \ref{prop:creal} and Lemma \ref{outsid}, we deduce that $f^{-n}(z) \in L$ for all integers $n \geq 1$.

{\bf Claim: $\lim_{n\to\infty} f^{-n}(z) = \theta$.}
One can prove by induction as in the proof of Proposition~\ref{prop:creal}~b) that there exists a sequence of  real numbers $c_n >1$ such that  for all integers $n \geq 0,$ we have $f^{-2n} (x, y)= (x_{2n}, y_{2n})$ satisfies
$\max \{x_{2n}, y_{2n}\} \geq c_n a_2 $
where $c_{n}= h(c_{n+1})$ for all integers $n \geq 0$ and $h(x)= a_2 (x^2-1)+1$.
Since $ 1< a_1 / a_2 < c_n$, we deduce that $(c_n)$ is decreasing and $\lim c_n= 1$.
Hence  $\lim f^{-2n} (x, y)= (a_2, a_2) = \theta$. Thus $ f^{-n} (x, y)$ converges to $ \theta$.

\vspace{0.5em}

\vspace{0.5em}

{\it Case 2: $z \in N$.}

In this case, we have $f^{-3n}(z)= (x_{3n}, y_{3n}) \in N,\; f^{-3n-1}(z) \in M$ and $ f^{-3n-2}(z) \in P$, for all integers $n \geq 0$.
We deduce as in the proof of item 2 of Proposition \ref{prop:creal}
that
$$x_{3(n-1)} < x_{3n}< -1,\; y_{3(n-1)} < y_{3n}<-1,\; \forall n \geq 1.$$
Let $l= \lim \; x_{3n} < -1$ and $l'=  lim \; y_{3n} < -1$.
By (\ref{hg}), this implies that
$$(ll'+c)l+c= l',\;  (ll'+c)l'+c= l.$$
Thus $l \in \{-1, a_1, a_2\}$. Thus $l =-1$.
Hence $f^{3n} (z)$ converges to $p= (-1, -1)$.
If $z \in M$, we deduce that $\lim f^{-3n}(z) = f( p)$.
If $z \in P$, then  $\lim f^{-3n}(z) = f^2( p)$.
\end{proof}

\begin{obs}

In \cite{HW}, a characterization of the Julia  $K(H_{a})= K^{+}\left(H_{a}\right) \cap K^{-}\left(H_{a}\right)$ 
of the H\'enon map $H_a(x, y) = (y, y^2+ax)$ where $0 < a < 1$ is given.
It is proved that the Julia set  
$K(H_{a})= \{\alpha, p\} \cup\left[ W^s (\alpha) \cap W^{u} (1-a, 1-a)\right] $, where
 $\alpha= (0,0)$ is the attracting fixed point of $H_a$ and   $p= (1-a,1-a)$ is the repelling fixed point of $H_a$. This served as motivation to show that the invariant sets $K^+$ and $K^-$ for the maps $f_c$ above, with  $0 < c < 1/4$, can be described as finite unions of stable  and unstable manifolds.
A further study, of the case $-1<c<0$ has been done by D.\ Caprio~\cite{Ca}.

\end{obs}

\vspace{0.5em}

\noindent {\bf Acknowledgements.}
The first author was supported by FAPESP grant 2011/12650-4.
The second author was supported by FAPESP grant 2011/16265-8.
The third author would like to express thanks to the  IME-USP (S\~ao Paulo) for the warm hospitality during his visit.
He was supported by FAPESP grants 2011/23199-1 and 2013/23643-4 and by CNPq grants 307154/2012-2 and 482519/2012-6.
He  would also like to express thanks to Pierre Arnoux and Patr\'{i}cia Cirilo Romano for fruitful discussions.

\bibliographystyle{amsalpha}
\bibliography{bib-fibonacci}

\end{document}